\documentclass[a4paper]{amsart}

\usepackage{amsmath,amssymb,amsthm}
\usepackage[all]{xy}
\usepackage{eucal}
\usepackage[colorlinks=true,backref=page]{hyperref}

\newtheorem{theorem}{Theorem}[section]
\newtheorem{proposition}[theorem]{Proposition}
\newtheorem{lemma}[theorem]{Lemma}
\newtheorem{corollary}[theorem]{Corollary}

\theoremstyle{definition}
\newtheorem{definition}[theorem]{Definition}

\newtheorem{problem}[theorem]{Problem}

\theoremstyle{remark}
\newtheorem{remark}[theorem]{Remark}

\numberwithin{equation}{section}

\newcommand{\Z}{\mathbb{Z}}
\newcommand{\Q}{\mathbb{Q}}

\newcommand{\NE}{N\mathcal{E}}
\newcommand{\B}{\mathcal{B}}
\newcommand{\Map}{\operatorname{Map}}

\SelectTips{cm}{}

\title[Self-Closeness numbers of rational mapping spaces]{Self-Closeness numbers of rational mapping spaces}

\author[Yichen Tong]{Yichen Tong}
\address{Department of Mathematics, Kyoto University, Kyoto, 606-8502, Japan}
\email{tong.yichen.25m@st.kyoto-u.ac.jp}

\date{\today}



\begin{document}
	
	\maketitle
	\begin{abstract}
		For a closed connected oriented manifold $M$ of dimension $2n$, it was proved by M\o ller and Raussen that the components of the mapping space from $M$ to $S^{2n}$ have exactly two different rational homotopy types. However, since this result was proved by the algebraic models for the components, it is unclear whether other homotopy invariants distinguish their rational homotopy types or not. The self-closeness number of a connected CW complex is the least integer $k$ such that any of its self-map inducing an isomorphism in $\pi_*$ for $*\le k$ is a homotopy equivalence, and there is no result on the components of mapping spaces so far. For a rational Poincar\'e complex $X$ of dimension $2n$ with finite $\pi_1$, we completely determine the self-closeness numbers of the rationalized components of the mapping space from $X$ to $S^{2n}$ by using their Brown-Szczarba models. As a corollary, we show that the self-closeness number does distinguish the rational homotopy types of the components. Since a closed connected oriented manifold is a rational Poincar\'e complex, our result partially generalizes that of M\o ller and Raussen.
	\end{abstract}
	
	\section{Introduction}
	
	Given two spaces $X$ and $Y$, we can associate the mapping space $\Map(X,Y)$. It is a classical problem in algebraic topology to classify the homotopy types of the path-components of $\Map(X,Y)$ for given $X$ and $Y$. This problem dates back, at least to the work of Whitehead \cite{W} in 1946, and there are many classification results for specific $X$ and $Y$. For instance, \cite{Ha2,Hu,K,W} studied the problem for $X=Y=S^n$. More recently, the problem was intensively studied in \cite{KiT,KoT,Ma,Mi,Su,Ta,Ts} when $Y$ is the classifying space of a Lie group $G$, as the components are the classifying spaces of gauge groups by \cite{G}. There are also results in \cite{GMS,Yk} for other $X$ and $Y$. See \cite{Sm} and its references for more details.

	Let $M$ be a closed connected oriented manifold of dimension $n$. We recall a classification result on the mapping space $\Map(M,S^n)$. By the Hopf degree theorem, the mapping degree gives a one-to-one correspondence between the path-components of $\Map(M,S^n)$ and $\Z$. Let $\Map(M,S^n;k)$ denote the path-component of degree $k$. In \cite{Ha2}, for $M$ with vanishing first Betti number, Hansen classified the homotopy types of the components of $\Map(M,S^n)$ such that $\Map(M,S^n,k)$ and $\Map(M,S^n;l)$ are homotopy equivalent if and only if one of the following conditions hold:
	\begin{enumerate}
		\item $|k|=|l|$ for $n$ even;
		
		\item the parity of $k$ and $l$ are equal for $n$ odd but 1,3,7;
		
		\item any $k$ and $l$ for $n=1,3,7$.
	\end{enumerate}
	Now we consider the rational homotopy types of $\Map(M,S^n;k)$. Since the rationalization of an odd sphere is an H-space, all components $\Map(M,S^n;k)$ have the same rational homotopy type for $n$ odd. Since the degree $k$ self-map of $S^n$ is a rational homotopy equivalence for $k\ne 0$, all components $\Map(M,S^n;k)$ but $\Map(M,S^n;0)$ have the same rational homotopy type for $n$ even. Then it remains to show whether or not $\Map(M,S^n;0)$ and $\Map(M,S^n;1)$ have the same rational homotopy type. M{\o}ller and Raussen \cite{MR} proved that they are not of the same rational homotopy type by considering their algebraic models, instead of specific homotopy invariants such as homology.

	The \emph{self-closeness number} of a connected CW complex $X$, denoted by $\NE(X)$, is defined to be the least integer $k$ such that every self-map of $X$ inducing an isomorphism in the homotopy groups of dimension $\le k$ is a homotopy equivalence. The self-closeness number was introduced by Choi and Lee \cite{CL} in 2015, and there are several results on it \cite{CL,L,OY1,OY2,OY3,To,Yt}. However, there are few explicit computations. Li \cite{L} and Oda and Yamaguchi \cite{OY3} determined the self-closeness numbers for some special homogeneous spaces. Later, the author \cite{To} computed those for some non-simply-connected finite complexes, which covers the previous results on non-simply-connected spaces. So far, all explicit computations are only for finite complexes, and there is no result on the components of mapping spaces. In this paper, we consider:

	\begin{problem}
		\label{prob}
		Does the self-closeness number distinguish the rational homotopy types of $\Map(M,S^n;0)$ and $\Map(M,S^n;1)$ for $n$ even?
	\end{problem}

	A space $X$ is called a \emph{rational Poincar\'e complex} of dimension $n$ if it is a finite complex of dimension $n$ such that $H^n(X;\Z)\cong\Z$ and the map
	\[
	H^i(X;\Z)\to H^{n-i}(X;\Z),\quad x\mapsto w\frown x
	\]
	is an isomorphism after tensoring with $\Q$, where $w$ is a generator of $H^n(X;\Z)$. Clearly, a closed connected oriented manifold of dimension $n$ is a rational Poincar\'e complex of dimension $n$. Moreover, we can consider Problem \ref{prob} for a rational Poincar\'e complex of dimension $n$, instead of a manifold $M$.

	To state the main result, we set notations. For a graded algebra $A$, let $QA^i$ denote its module of indecomposables of degree $i$. Let $X$ be a rational Poincar\'e complex of dimension $2n$. We say that $X$ is \emph{primitive} if $QH^i(X;\Q)=H^i(X;\Q)$ for $i<2n$, as $H_i(X;\Q)$ consists of primitive homology classes for $i<2n$ with respect to the comultiplication of $H_*(X;\Q)$ induced by the diagonal map of $X$. We define $d(X)$ to be the least integer $d$ such that $H^d(X;\Q)\ne 0$ and $d\ge n$. Let $Y_{(0)}$ denote the rationalization of a nilpotent space $Y$ in the sense of \cite{HMR}. Now we state the main theorem.

	\begin{theorem}
		\label{main}
		Let $X$ be a rational Poincar\'e complex of dimension $2n$ with finite $\pi_1$. Then we have
		\[
		\NE(\Map(X,S^{2n};1)_{(0)})=4n-1
		\]
		and
		\[
		\NE(\Map(X,S^{2n};0)_{(0)})=
		\begin{cases}
			2n&X\text{ is primitive,}\\
			d(X)&X\text{ is not primitive}.
		\end{cases}
		\]
	\end{theorem}

	We immediately get the following corollary, which gives a positive answer to Problem \ref{prob}.

	\begin{corollary}
		Let $X$ be a rational Poincar\'e complex of dimension $2n$ with finite $\pi_1$. The self-closeness number distinguishes the rational homotopy types of $\Map(X,S^{2n};k)$ for $k=0,1$.
	\end{corollary}

	This paper is organized as follows. In Section 2, we introduce the self-closeness number of a minimal Sullivan algebra, and prove that the self-closeness number of a simply-connected rational space coincides with that of its minimal model. In Section 3, we recall the Brown-Szczarba models for $\Map(X,S^{2n})$ and its components, where $X$ is a rational Poincar\'e complex of dimension $2n$ with finite $\pi_1$. In Sections 4 and 5, we compute the self-closeness numbers of the minimal models for $\Map(X,S^{2n};k)_{(0)}$ for $k=0,1$.

	\subsection*{Acknowledgements}
	I would like to thank Daisuke Kishimoto for encouraging me on this research, and for mentioning operations on manifolds such as handle attachments and surgeries, which motivated the filtration of the vector space in Section 5. I am also grateful for his help in improving the notations and distribution of the manuscript to make it readable. I would also like to thank Toshihiro Yamaguchi for valuable discussions and suggestions. I also appreciate Katsuhiko Kuribayshi, Takahito Naito, Masahiro Takeda and Shun Wakatsuki for discussions in the early stage of this research, and the referees for carefully reading the manuscript and for advice on revision. This work was supported by JST SPRING, Grant Number JPMJSP2110.

	
	\section{Algebraic self-closeness number}
	
	In this section, we define the self-closeness number of a minimal Sullivan algebra, and prove that it coincides with the self-closeness number of a corresponding rational space. Hereafter, we will assume that all algebras and vector spaces will be over the field $\Q$.

	We say that a dga $A$ is an algebraic model for a space $X$ if there is a zig-zag of quasi-isomorphisms
	\[
	A\xrightarrow{\simeq}A_1\xleftarrow{\simeq}A_2\xrightarrow{\simeq}\cdots\xleftarrow{\simeq}A_n\xrightarrow{\simeq}A_{\rm PL}(X)
	\]
	where $A_{\rm PL}(X)$ denotes the dga of piecewise linear forms on $X$. If $X$ is a nilponent space, then the rationalization $X\to X_{(0)}$ is a rational homotopy equivalence, implying that a dga $A$ is an algebraic model for $X$ if and only if so is for $X_{(0)}$. We recall a minimal model for a space. Let $V$ be a positively graded vector space, and let $\Lambda V$ denote the free commutative graded algebra generated by $V$. We say that a dga is a minimal Sullivan algebra if it is of the form $(\Lambda V,d)$ such that
	\[
	d(V_n)\subset\Lambda(V_{\le{n-1}})
	\]
	for all $n\ge 1$, where $V_n$ and $V_{\le n-1}$ are the degree $n$ and degree $\le n-1$ parts of $V$, respectively. If an algebraic model for a space $X$ is a minimal Sullivan algebra, then we call it a minimal model for $X$. If $(\Lambda V,d)$ is a minimal model for a space $X$, then a zig-zag of quasi-isomorphisms patch together to yield a quasi-isomorphism
	\[
	(\Lambda V,d)\xrightarrow{\simeq}A_{\rm PL}(X).
	\]
	We state the fundamental theorem of rational homotopy theory.

	\begin{theorem}
		\label{fundamental}
		For every simply-connected rational space $X$ of finite rational type, there is a minimal model
		\[
		\alpha\colon(\Lambda V,d)\xrightarrow{\simeq}A_{\rm PL}(X)
		\]
		satisfying the following properties.
		
		\begin{enumerate}
			\item $(\Lambda V,d)$ is unique, up to isomorphism.
			
			\item The quasi-isomorphism $\alpha$ is natural, up to homotopy.
			
			\item There is a natural isomorphism
			\[
			V\cong\mathrm{Hom}(\pi_*(X),\Q).
			\]
		\end{enumerate}
	\end{theorem}
	If $X$ is a nilpotent space of finite type such that $X_{(0)}$ is simply-connected, then we often say that a minimal model for $X_{(0)}$ is a minimal model for $X$. For a map $f\colon\Lambda V\to\Lambda W$ between free commutative graded algebras, we define its linear part $f_0\colon V\to W$ by the composite
	\[
	V\xrightarrow{\rm incl}\Lambda V\xrightarrow{f}\Lambda W\xrightarrow{\rm proj}\Lambda W/(\Lambda^+W)^2\cong W
	\]
	where $\Lambda^+W$ denotes the ideal of $\Lambda W$ generated by elements of positive degrees. We define the self-closeness number of a minimal Sullivan algebra.

	\begin{definition}
		The self-closeness number of a minimal Sullivan algebra $(\Lambda V,d)$, denoted by $\NE(\Lambda V,d)$, is the least integer $n$ such that any dga map $(\Lambda V,d)\to(\Lambda V,d)$ is an isomorphism whenever its linear part is an isomorphism in degree $\le n$.
	\end{definition}

	We will use the following lemma to compute the self-closeness number of a rational space.

	\begin{proposition}
		\label{NE(minimal)}
		Let $X$ be a simply-connected rational CW complex of finite rational type, and let $(\Lambda V,d)$ be its minimal model. Then we have
		\[
		\NE(X)=\NE(\Lambda V,d).
		\]
	\end{proposition}
	
	\begin{proof}
		Let $f\colon X\to X$ be a map, and let $g\colon(\Lambda V,d)\to(\Lambda V,d)$ be a dga map corresponding to a map $f$, which exists, up to homotopy, by Theorem \ref{fundamental}. Since there is a natural isomorphism $V\cong\mathrm{Hom}(\pi_*(X),\Q)$ as in Theorem \ref{fundamental}, there is a commutative diagram
		\[
		\xymatrix{
			V\ar[r]^{g_0}\ar[d]_\cong&V\ar[d]^\cong\\
			\mathrm{Hom}(\pi_*(X),\Q)\ar[r]^{(f_*)^*}&\mathrm{Hom}(\pi_*(X),\Q).
		}
		\]
		Then $g_0$ is an isomorphism in degree $k$ if and only if $f_*\colon\pi_k(X)\to\pi_k(X)$ is an isomorphism. Thus if $\NE(X)=m$, then $g$ is an isomorphism whenever $g_0$ is an isomorphism in degree $\le m$, that is, $\NE(\Lambda V,d)\le m$. On the other hand, if $\NE(\Lambda V,d)=n$, then $f$ is an isomorphism in $\pi_*$ whenever it is an isomorphism in $\pi_*$ for $*\le n$, implying $\NE(X)\le n$ by the J.H.C. Whitehead theorem. Thus the proof is finished.
	\end{proof}

	For the rest of this section, let $X$ denote a rational Poincar\'e complex of dimension $2n$ with finite $\pi_1$. Then by \cite[Theorem 2.5]{HMR}, the mapping space $\Map(X,S^{2n};k)$ is a nilpotent complex of finite type. Then to apply Proposition \ref{NE(minimal)} to $\Map(X,S^{2n};k)_{(0)}$, we need the following lemma.

	\begin{lemma}
		\label{1-conn}
		For any integer $k$, $\Map(X,S^{2n};k)_{(0)}$ is simply-connected.
	\end{lemma}
	
	\begin{proof}
		By \cite[Theorem 3.11]{HMR}, there is a homotopy equivalence
		\[
		\Map(X,S^{2n};k)_{(0)}\simeq\Map(X,S^{2n}_{(0)};r\circ k)
		\]
		where $r\colon S^{2n}\to S^{2n}_{(0)}$ denotes the rationalization. Since there is a $(4n-1)$-equivalence $i\colon S^{2n}_{(0)}\to K(\Q,2n)$, we get a $(2n-1)$-equivalence
		\[
		i_*\colon\Map(X,S^{2n}_{(0)};r\circ k)\to\Map(X,K(\Q,2n);i\circ r\circ k)
		\]
		and since $K(\Q,2n)$ is an H-group, we have
		\[
		\Map(X,K(\Q,2n);i\circ r\circ k)\simeq\Map(X,K(\Q,2n);0).
		\]
		On the other hand, by the theorem of Thom \cite{Th}, there is a homotopy equivalence
		\[
		\Map(X,K(\Q,2n);0)\simeq\prod_{k=0}^{2n-1}K(H^k(X;\Q),2n-k).
		\]
		Since $\pi_1(X)$ is finite, we have $H^{2n-1}(X;\Q)\cong H^1(X;\Q)=0$ by the definition of a rational Poincar\'e complex. Thus the proof is finished.
	\end{proof}

	
	\section{Brown-Szczarba model}
	
	In this section, we give an algebraic model for $\Map(X,S^{2n})$ of Brown and Szczarba \cite{BS} for a general space $X$, and specialize it to the case when $X$ is a rational Poincar\'e complex of dimension $2n$. We also give algebraic models for the components of $\Map(X,S^{2n})$. We will write the rational homology and cohomology of $X$ simply by $H_*(X)$ and $H^*(X)$, respectively.

	Brown and Szczarba \cite[Theorem 5.3]{BS} proved that for a CW complex $X$ and a nilpotent complex $Y$ of finite type, there is an algebraic model for $\Map(X,Y)$ of the form
	\[
	(\Lambda(V\otimes H_*(X)),d)
	\]
	where $(\Lambda V,d)$ is a minimal model for $Y$ and we set
	\[
	|x\otimes y|=|x|-|y|
	\]
	for $x\in V$ and $y\in H_*(X)$, where $|a|$ denotes the degree of an element $a$ in a graded vector space. The differential is defined in terms of the differential of the minimal model $(\Lambda V,d)$ and the comultiplication of the chain complex of $X$, instead of the comultiplication of $H_*(X)$ in general. We specialize this algebraic model to $\Map(X,S^{2n})$. Recall that $S^{2n}$ has a minimal model given by
	\[
	(\Lambda(u,v),d),\quad du=0,\quad dv=u^2
	\]
	where $|u|=2n$ and $|v|=4n-1$. Let $V$ be a graded vector space spanned by $u,v$. Then there is an algebraic model for $\Map(X,S^{2n})$ of the form
	\begin{equation}
		\notag
		\label{BS model}
		(\Lambda(V\otimes H_*(X)),d)
	\end{equation}
	where the differential is defined in terms of the comultiplication of the chain complex of $X$. However, the proof of \cite[Lemma 5.1]{BS} implies that in our special case, the differential is actually defined in terms of the comultiplication of $H_*(X)$ as follows.

	\begin{theorem}
		\label{model}
		There is an algebraic model for $\Map(X,S^{2n})$ of the form
		\[
		(\Lambda(V\otimes H_*(X)),d)
		\]
		such that
		\[
		d(u\otimes x)=0\quad\text{and}\quad d(v\otimes x)=\sum_i(u\otimes y_i)(u\otimes z_i)
		\]
		where $\Delta(x)=\sum_iy_i\otimes z_i$ for the comultiplication $\Delta$ of $H_*(X)$.
	\end{theorem}

	Hereafter, let $X$ denote a rational Poincar\'e complex of dimension $2n$ with finite $\pi_1$. We specialize the algebraic model in Theorem \ref{model}. First, for $x,y,z\in H_*(X)$, we define a rational number $\epsilon(x,y,z)$ by
	\[
	\epsilon(x,y,z)=\langle y^*\smile x^*,z\rangle
	\]
	where $x^*,y^*\in H^*(X)$ denote the dual cohomology classes of $x,y$ and $\langle-,-\rangle$ denotes the pairing of cohomology and homology classes. Note that
	\begin{equation}
		\label{epsilon degree}
		|x|+|y|=|z|\quad\text{whenever}\quad\epsilon(x,y,z)\ne 0.
	\end{equation}

	Next, we choose a basis of $H_*(X)$ as follows. Hereafter, we fix a generator $w$ of $H_{2n}(X)\cong\Q$. Let
	\[
	\B_0=\{1\}\subset H_0(X)\quad\text{and}\quad\B_{2n}=\{w\}\subset H_{2n}(X).
	\]
	We choose any basis $\B_i$ of $H_i(X)$ for $i=2,3,\ldots,n-1$, where $H_1(X)=0$ since $\pi_1(X)$ is finite. By definition, the cup product
	\[
	H^i(X)\otimes H^{2n-i}(X)\to H^{2n}(X)\cong\Q,\quad x\otimes y\mapsto x\smile y
	\]
	is nondegenerate. Then for each $x\in\B_i$ with $i=2,3,\ldots,n-1$, there is unique $\mathrm{PD}(x)\in H_{2n-i}(X)$ satisfying that for $y\in\B_i$,
	\begin{equation}
		\notag
		\epsilon(y,\mathrm{PD}(x),w)=
		\begin{cases}
			1&y=x,\\
			0&y\ne x.
		\end{cases}
	\end{equation}
	We set $\B_{2n-i}=\{\mathrm{PD}(x)\mid x\in\B_i\}$ for $i=2,3,\ldots,n-1$. Then $\B_{2n-i}$ is a basis of $H_{2n-i}(X)$ for $i=2,3,\ldots,n-1$. Since $H_{2n-1}(X)\cong H_1(X)=0$, it remains to choose a basis for $H_n(X)$. Since the cup product
	\[
	H^n(X)\otimes H^n(X)\to H^{2n}(X)\cong\Q
	\]
	is a nondegenerate symmetric bilinear form on $H^n(X)$ for $n$ even and a nondegenerate anti-symmetric bilinear form on $H^n(X)$ for $n$ odd, we can choose a basis $\B_n=\{x_1,\ldots,x_{b_n}\}$ of $H_n(X)$ such that the matrix $(\epsilon(x_i,x_j,w))_{1\le i,j\le b_n}$ is a regular diagonal matrix for $n$ even and of the form
	\[
	\begin{pmatrix}
		0&-\lambda_1\\
		\lambda_1&0\\
		&&\ddots\\
		&&&0&-\lambda_{\frac{b_n}{2}}\\
		&&&\lambda_{\frac{b_n}{2}}&0
	\end{pmatrix}
	\]
	for $n$ odd, where $\lambda_1\cdots\lambda_{\frac{b_n}{2}}\ne 0$ and $b_i=\dim H_i(X)$. We set
	\[
	\B=\B_0\sqcup\cdots\sqcup\B_{2n}.
	\]
	Then $\B$ is a basis of $H_*(X)$. By definition, we have
	\begin{equation*}
		\label{comultiplication}
		\Delta(x)=\sum_{x_1,x_2\in\B}\epsilon(x_1,x_2,x)x_1\otimes x_2.
	\end{equation*}
	for $x\in\B$, where $\epsilon(x_1,x_2,x)=0$ unless $|x_1|+|x_2|=|x|$ by \eqref{epsilon degree}. By Theorem \ref{model}, we get:

	\begin{theorem}
		\label{modelM}
		There is an algebraic model for $\Map(X,S^{2n})$ of the form
		\[
		(\Lambda(V\otimes H_*(X)),d)
		\]
		such that for $x\in\B$,
		\begin{align*}
			d(u\otimes x)&=0,\\
			d(v\otimes x)&=\sum_{x_1,x_2\in\B}\epsilon(x_1,x_2,x)(u\otimes x_1)(u\otimes x_2).
		\end{align*}
	\end{theorem}

	Now we give an algebraic model for the component $\Map(X,S^{2n};k)$. Note that the degree zero part of $\Lambda(V\otimes H_*(X))$ is spanned by $u\otimes w$ and the unit $1$. Let $I_k$ be the ideal of $\Lambda(V\otimes H_*(X))$ generated by the element $u\otimes w-k$ of degree zero. Since $d(u\otimes w-k)=0$, $I_k$ is an ideal of dga. In particular, we can define the quotient dga $(\Lambda(V\otimes H_*(X))/I_k,d)$. By \cite[Theorem 6.1]{BS}, we get:

	\begin{theorem}
		\label{component model}
		The quotient dga $(\Lambda(V\otimes H_*(X))/I_k,d)$ is an algebraic model for $\Map(X,S^{2n};k)$.
	\end{theorem}

	\begin{remark}
		Let $Y$ be a connected CW complex of dimension at most $4n-2$. Then there is a one-to-one correspondence between $H^{2n}(Y;\Z)$ and path-components of $\Map(Y,S^{2n})$. For each $\lambda\in H^{2n}(Y;\Z)$, M\o ller and Raussen \cite[Propositions 2.3, 2.4]{MR} constructed a minimal model for $\Map(Y,S^{2n};\lambda)$ by using the rational Postnikov tower of $S^{2n}$. We will construct minimal models for $\Map(X,S^{2n};k)$ with $k=0,1$ in Corollary \ref{model k=0} and Proposition \ref{model 1}, which are the special cases of M\o ller and Raussen \cite[Propositions 2.3, 2.4]{MR} by Theorem \ref{fundamental}. However, our minimal models are more explicit so that we can determine the self-closeness numbers of $\Map(X,S^{2n};k)_{(0)}$ from them.
	\end{remark}
	We show some properties of $\epsilon(x,y,z)$ that we will use later.
	
	\begin{lemma}
		\label{comm}
		For $x,y,z\in H_*(X)$, we have:
		\[
		\epsilon(x,1,x)=\epsilon(1,x,x)=1,\quad\epsilon(x,y,z)=(-1)^{|x||y|}\epsilon(y,x,z).
		\]
	\end{lemma}
	
	\begin{proof}
		The first identity follows from
		\[
		\epsilon(x,1,x)=\langle 1\smile x^*,x\rangle=1=\langle x^*\smile 1,x\rangle=\epsilon(1,x,x).
		\]
		The second identity follows from
		\[
		\epsilon(x,y,z)=\langle y^*\smile x^*,z\rangle=(-1)^{|x||y|}\langle x^*\smile y^*,z\rangle=\epsilon(y,x,z)
		\]
		completing the proof.
	\end{proof}

	\begin{lemma}
		\label{asso}
		For $x_1,x_2,x_3,x_4\in H_*(X)$, we have:
		\[
		\sum_{y\in\B}\epsilon(x_1,x_2,y)\epsilon(y,x_3,x_4)=\sum_{z\in\B}\epsilon(x_2,x_3,z)\epsilon(x_1,z,x_4).
		\]
	\end{lemma}
	\begin{proof}
		By definition, we have
		\[
		x_2^*\smile x_1^*=\sum_{y\in\B}\epsilon(x_1,x_2,y)y^*.
		\]
		Then we get
		\[
		\langle x_3^*\smile x_2^*\smile x_1^*,x_4\rangle=\sum_{y\in\B}\epsilon(x_1,x_2,y)\langle x_3^*\smile y^*,x_4\rangle=\sum_{y\in\B}\epsilon(x_1,x_2,y)\epsilon(y,x_3,x_4).
		\]
		Quite similarly, we can get
		\[
		\langle x_3^*\smile x_2^*\smile x_1^*,x_4\rangle=\sum_{z\in\B}\epsilon(x_2,x_3,z)\epsilon(x_1,z,x_4).
		\]
		Thus the statement is proved.
	\end{proof}

	
	\section{Degree zero component}
	
	In this section, we compute the self-closeness number of $\Map(X,S^{2n};0)_{(0)}$ by using the algebraic model in Theorem \ref{component model}. Let
	\begin{equation}
		\notag
		\widehat{\B}=\{x\in\B\mid 0<|x|<2n\}.
	\end{equation}
	For a graded set $S$, let $\langle S\rangle$ denote the graded vector space spanned by $S$. We define a graded vector space by
	\begin{equation}
		\label{W}
		W=\langle u\otimes x,v\otimes x,v\otimes w\mid x\in\B_0\sqcup\widehat{\B}\rangle
	\end{equation}
	and a dga $(\Lambda W,d)$ by
	\begin{align*}
		d(u\otimes x)&=0,\\
		d(v\otimes x)&=\sum_{x_1,x_2\in\B}\epsilon(x_1,x_2,x)(u\otimes x_1)(u\otimes x_2),\\
		d(v\otimes w)&=\sum_{x_1,x_2\in\widehat{\B}}\epsilon(x_1,x_2,w)(u\otimes x_1)(u\otimes x_2)
	\end{align*}
	for $x\in\B_0\sqcup\widehat{\B}$.

	\begin{corollary}
		\label{model k=0}
		The dga $(\Lambda W,d)$ is the minimal model for $\Map(X,S^{2n};0)$.
	\end{corollary}
	
	\begin{proof}
		By definition, the term $u\otimes w$ is not included in any differential of the quotient dga $(\Lambda(V\otimes H_*(X))/I_0,d)$ in Theorem \ref{component model}, and so $(\Lambda(V\otimes H_*(X))/I_0,d)$ is isomorphic with $(\Lambda W,d)$. Clearly, $(\Lambda W,d)$ is minimal, and the statement is proved.
	\end{proof}

	By Lemmas \ref{NE(minimal)} and \ref{1-conn}, we aim to compute the self-closeness number of $(\Lambda W,d)$ for determining that of $\Map(X,S^{2n};0)_{(0)}$. Let $f\colon(\Lambda W,d)\to(\Lambda W,d)$ be a dga map. We give a matrix representation of $f_0$. For $0\le k<2n$, we equip $\B_k$ with any total order such that
	\[
	\B_k=\{x_1^k<\cdots<x_{b_k}^k\}.
	\]
	Let $W_k$ denote the degree $k$ part of $W$. Then we have
	\[
	W=\left(\bigoplus_{k=0}^{2n-2}W_{2n-k}\right)\oplus\left(\bigoplus_{k=0}^{2n}W_{4n-k-1}\right).
	\]
	For $0\le k\le 2n-2$, we have
	\[
	W_{2n-k}=\langle u\otimes x\mid x\in\B_k\rangle\quad\text{and}\quad W_{4n-k-1}=\langle v\otimes x\mid x\in\B_k\rangle.
	\]
	Moreover, we have $W_{2n-1}=\langle v\otimes w\rangle$. For $0\le k\le 2n-2$, let $A_k(f)$ denote the matrix representation of $f_0\colon W_{2n-k}\to W_{2n-k}$ with respect to the ordered basis $\B_{2n-k}$. For $0\le k\le 2n$, let $B_k(f)$ denote the matrix representation of $f_0\colon W_{4n-k-1}\to W_{4n-k-1}$ with respect to the ordered basis $\B_{4n-k-1}$. Let $C_{ij}$ denote the $(i,j)$-entry of a matrix $C$. Then by definition, we have
	\[
	f_0(u\otimes x_j^k)=\sum_{i=1}^{b_k}A_k(f)_{ij}u\otimes x_i^k\quad\text{and}\quad f_0(v\otimes x_j^k)=\sum_{i=1}^{b_k}B_k(f)_{ij}v\otimes x_i^k
	\]
	for $1\le j\le b_k$. Obviously, a dga map $f$ is an isomorphism if and only if its linear part $f_0$ is an isomorphism. Then the following lemma is immediate from the definition of $A_k(f)$ and $B_k(f)$.

	\begin{lemma}
		\label{upper bound by mat}
		The following are equivalent:
		\begin{enumerate}
			\item $\NE(\Lambda W,d)\le m$;
			
			\item For any dga map $f\colon(\Lambda W,d)\to(\Lambda W,d)$, if $A_k(f)$ and $B_l(f)$ are regular for $2n-m\le k\le 2n-2$ and $4n-m-1\le l\le 2n$, then $A_k(f)$ and $B_l(f)$ are regular for all $0\le k\le 2n-2$ and $0\le l\le 2n$.
		\end{enumerate}
	\end{lemma}

	For $x\in\B_i$ and $k\le i$, we also define an $b_k\times b_{i-k}$ matrix $E_k(x)$ by
	\[
	E_k(x)_{pq}=\epsilon(x_p^k,x_q^{i-k},x).
	\]
	By our choice of the basis $\B$ and the definition of a rational Poincar\'e complex, we have:

	\begin{lemma}
		\label{E(w)}
		The matrix $E_k(w)$ is regular for $0\le k\le 2n$.
	\end{lemma}

	We prove relations among $A_k(f),B_k(f)$ and $E_k(x)$.

	\begin{lemma}
		\label{mat}
		Let $f\colon(\Lambda W,d)\to(\Lambda W,d)$ be a dga map. For each $x^i_p\in\B_i$, we have
		\[
		A_k(f)E_k(x_p^i)A_{i-k}(f)^\textrm{T}=\sum_{a=1}^{b_i}B_i(f)_{ap}E_k(x_a^i)
		\]
		where $0\le k\le i$ for $i<2n$ and $0<k<2n$ for $i=2n$.
	\end{lemma}
	
	\begin{proof}
		We only prove the $i<2n$ case, and the $i=2n$ case can be proved verbatim, where the only difference is the range of $k$. Let $i<2n$. The quadratic part of $f(d(v\otimes x_p^i))$ is
		\begin{align*}
			&\sum_{k=0}^i\sum_{r=1}^{b_k}\sum_{t=1}^{b_{i-k}}\epsilon(x_r^k,x_t^{i-k},x_p^i)f_0(u\otimes x_r^k)f_0(u\otimes x_t^{i-k})\\
			&=\sum_{k=0}^i\sum_{q,r=1}^{b_k}\sum_{s,t=1}^{b_{i-k}}\epsilon(x_r^k,x_t^{i-k},x_p^i)A_k(f)_{qr}A_{i-k}(f)_{st}(u\otimes x_q^k)(u\otimes x_s^{i-k})
		\end{align*}
		and the quadratic part of $d(f(v\otimes x_p^i))$ is
		\begin{align*}
			d(f_0(v\otimes x_p^i))&=d\left(\sum_{a=1}^{b_i}B_i(f)_{ap}v\otimes x_a^i\right)\\
			&=\sum_{a=1}^{b_i}\sum_{k=0}^i\sum_{q=1}^{b_k}\sum_{s=1}^{b_{i-k}}\epsilon(x_q^k,x_s^{i-k},x_a)B_i(f)_{ap}(u\otimes x_q^k)(u\otimes x_s^{i-k}).
		\end{align*}
		Then since $f(d(v\otimes z))=d(f(v\otimes z))$ for $z\in\B$, we get
		\[
		\sum_{r=1}^{b_k}\sum_{t=1}^{b_{i-k}}\epsilon(x_r^k,x_t^{i-k},x_p^i)A_k(f)_{qr}A_{i-k}(f)_{st}=\sum_{a=1}^{b_i}\epsilon(x_q^k,x_s^{i-k},x_a)B_i(f)_{ap}
		\]
		for fixed $q=1,...,b_k$, $s=1,...,b_{i-k}$ and $k=0,...,i$. Thus the statement is proved.
	\end{proof}

	\begin{corollary}
		\label{U&V}
		Let $f\colon(\Lambda W,d)\to(\Lambda W,d)$ be a dga map. For $0\le i<2n$, we have
		\[
		B_i(f)=A_0(f)A_i(f)
		\]
	\end{corollary}
	
	\begin{proof}
		Clearly, the matrix $(E_i(x_1^i),\ldots,E_i(x_{b_i}^i))$ is the $b_i\times b_i$ identity matrix, and so by Lemma \ref{mat}, we have
		\[
		A_i(f)A_0(f)^\textrm{T}=A_i(f)(E_i(x_1^i),\ldots,E_i(x_{b_i}^i))A_0(f)^\textrm{T}=B_i(f).
		\]
		Since $A_0(f)$ is a $1\times 1$ matrix, the proof is finished.
	\end{proof}

	We will use the following property of $d(X)$.

	\begin{lemma}
		\label{d(M)}
		For $x\in\B$, if $2n-|x|<d(X)$, then $|x|>n$.
	\end{lemma}
	\begin{proof}
		If $d(X)=n$, we have $|x|>2n-d(X)=n$. If $d(X)>n$, we have $|x|\ge d(X)>n$ since by definition, $\B_i=\varnothing$ for $2n-d(X)<i<d(X)$. In either case we have $|x|>n$, completing the proof.
	\end{proof}

	Now we are ready to compute the self-closeness number of $(\Lambda W,d)$.

	\begin{proposition}
		\label{NE 1}
		If $X$ is not primitive, then
		\[
		\NE(\Lambda W,d)=d(X).
		\]
	\end{proposition}
	\begin{proof}
		First, we prove $\NE(\Lambda W,d)\le d(X)$. Let $f\colon(\Lambda W,d)\to(\Lambda W,d)$ be a dga map such that $f_0$ is an isomorphism in degrees $\le d(X)$. Then $A_k(f)$ is regular for $2n-k\le d(X)$. By Lemma \ref{upper bound by mat}, it is sufficient to show that $A_k(f)$ and $B_l(f)$ are regular for $0\le k<2n$ and $0\le l\le 2n$. To this end, we take two steps.

		\noindent\textbf{Step 1.} By Lemma \ref{mat}, for $x_p^i=w$, we have
		\begin{equation}
			\label{top}
			A_k(f)E_k(w)A_{2n-k}(f)^\textrm{T}=B_{2n}(f)E_k(w)
		\end{equation}
		for $0<k<2n$, where $B_{2n}(f)$ is a $1\times 1$ matrix. By definition, $b_{d(X)}=b_{2n-d(X)}\ne0$, implying $W_{d(X)}$ and $W_{2n-d(X)}$ are non-trivial. By assumption, $A_{2n-d(X)}(f)$ is regular. Moreover, since $2n-d(X)\le d(X)$, $A_{d(X)}(f)$ is regular too. Then the $1\times 1$ matrix $B_{2n}(f)$ is regular by Lemma \ref{E(w)} and \eqref{top} for $k=d(X)$. This implies that $A_k(f)$ is regular for $0<k<2n$ too by Lemma \ref{E(w)} and \eqref{top}.

		\noindent\textbf{Step 2.} By assumption, there is $0<i<2n$ such that $QH^i(X)\ne H^i(X)$, implying $E_j(x_p^i)$ is non-trivial for some $x_p^i\in\B_j$ and $0<j<i$. This also implies that $W_{2n-i}$, $W_{2n-j}$ and $W_{2n-i+j}$ are non-trivial. On the other hand, by Lemma \ref{mat}, we have
		\[
		A_j(f)E_j(x_p^i)A_{i-j}(f)^\textrm{T}=\sum_{a=1}^{b_i}B_i(f)_{ap}E_j(x_a^i).
		\]
		We have seen in Step 1 that both $A_j(f)$ and $A_{i-j}(f)$ are regular since $0<j,i-j<2n$. Then $A_j(f)E_j(x_p^i)A_{i-j}(f)^\textrm{T}$ is non-trivial, implying that $B_i(f)$ is non-trivial too. Then since $B_i(f)$ is non-trivial and $A_0(f)$ is a $1\times 1$ matrix, it follows from Corollary \ref{U&V} that $A_0(f)$ is regular. Thus we have obtained that $A_k(f)$ is regular for $0\le k<2n$ by Step 1. Moreover, by Corollary \ref{U&V}, $B_l(f)$ is regular for $0\le l<2n$ too, and so we obtained that $B_l(f)$ is regular for $0\le l\le 2n$. Thus we obtain $\NE(\Lambda W,d)\le d(X)$.

		Next, we prove $\NE(\Lambda W,d)\ge d(X)$. Consider a self-map of a commutative graded algebra $g\colon\Lambda W\to\Lambda W$
		such that $g(v\otimes w)=0$, $g(v\otimes x)=0$ and
		\[
		g(u\otimes x)=
		\begin{cases}
			u\otimes x  &2n-|x|<d(X),\\
			0  &2n-|x|\ge d(X)
		\end{cases}
		\]
		for $x\in\B_0\sqcup\widehat{\B}$. Then $g_0$ is an isomorphism in degrees $<d(X)$ and trivial in degree $d(X)$, where $W_{d(X)}$ is non-trivial. By definition, we have
		\[
		dg(u\otimes x)=0=g(d(u\otimes x))
		\]
		for $x\in\B_0\sqcup\widehat{\B}$. For $x\in\B$, $d(v\otimes x)$ is a linear combination of $(u\otimes x_1)(u\otimes x_2)$ such that $|x_1|+|x_2|=|x|$. If $2n-|x_1|<d(X)$ and $2n-|x_2|<d(X)$, we have $|x|=|x_1|+|x_2|>2n$ by Lemma \ref{d(M)}, which is impossible. Thus either $2n-|x_1|\ge d(X)$ or $2n-|x_2|\ge d(X)$, implying that
		\[
		dg(v\otimes x)=0=g(d(v\otimes x))
		\]
		for $x\in\B$. Thus $g$ is a dga map, implying $\NE(\Lambda W,d)\ge d(X)$. Therefore the proof is finished.
	\end{proof}

	\begin{proposition}
		\label{NE 2}
		If $X$ is primitive, then
		\[
		\NE(\Lambda W,d)=2n.
		\]
	\end{proposition}
	\begin{proof}
		Let $f\colon(\Lambda W,d)\to(\Lambda W,d)$ be a dga map. Suppose that $f_0$ is an isomorphism in degrees $\le 2n$. Then $A_k(f)$ and $B_{2n}(f)$ are regular for $0\le k<2n$. So by Corollary \ref{U&V}, $B_k(f)$ is regular for $0\le k<2n$ too. Thus by Lemma \ref{upper bound by mat}, we obtain $\NE((\Lambda W,d))\le 2n$.

		Consider a self-map of a commutative graded algebra $g\colon\Lambda W\to\Lambda W$
		given by
		\[
		g(u\otimes x)=
		\begin{cases}
			0&|x|=0,\\
			u\otimes x&0<|x|<2n
		\end{cases}
		\quad\text{and}\quad
		g(v\otimes x)=
		\begin{cases}
			0&0\le|x|<2n,\\
			v\otimes x&|x|=2n.
		\end{cases}
		\]
		Then we have
		\[
		dg(u\otimes x)=0=g(d(u\otimes x))
		\]
		for $0\le|x|<2n$. Since $QH^i(X)=H^i(X)$ for $i<2n$, we have $\epsilon(x_1,x_2,x)=0$ for $0\le|x|<2n$ unless $x_1=1$ or $x_2=1$. Then we have
		\[
		d(v\otimes x)=
		\begin{cases}
			(u\otimes1)^2&|x|=0,\\
			2(u\otimes1)(u\otimes x)&0<|x|<2n,
		\end{cases}
		\]
		implying $d(g(v\otimes x))=0=g(d(v\otimes x))$ for $0\le|x|<2n$. By definition, $d(v\otimes w)$ is a linear combination of $(u\otimes x_1)(u\otimes x_2)$ for $0<|x_1|,|x_2|<2n$. Then we have
		\[
		dg(v\otimes w)=d(v\otimes w)=g(d(v\otimes w)).
		\]
		Thus $g$ is a dga map. Clearly, $g_0$ is an isomorphism in degrees $<2n$ and trivial in degree $2n$, where $W_{2n}=\langle u\otimes1\rangle$ is non-trivial. Thus we get $\NE(\Lambda W,d)\ge 2n$, completing the proof.
	\end{proof}
	
		
		\section{Degree one component}
		
		In this section, we determine the self-closeness number of $\Map(X,S^{2n};1)_{(0)}$. To this end, we take two steps. In the first step, we construct the minimal model $(\Lambda\overline{W},d)$ for $\Map(X,S^{2n};1)_{(0)}$ by using its algebraic model in Theorem \ref{component model}, where $\overline{W}$ concentrates in degrees $\le 4n-1$. If there is an isomorphism of dgas
		\[
		(\Lambda U,d)\otimes(\Lambda(s),0)\xrightarrow{\cong}(\Lambda\overline{W},d)
		\]
		where $U$ is the degree $<4n-1$ part of $\overline{W}$ and $|s|=4n-1$, then the self-closeness number of $(\Lambda\overline{W},d)$ turns out to be $4n-1$. To prove the existence of such an isomorphism, it is sufficient to show that a certain element in $(\Lambda U,d)$ of degree $4n$ is a coboundary. In the second step, we show that the sum of the above element of degree $4n$ and a certain coboundary belongs to a vector subspace of $\Lambda U$ having a direct sum decomposition. Then we show that the above sum is trivial in each direct summand, implying that the above element of degree $4n$ turns out to be a coboundary.

		\subsection{Minimal model}
		
		The algebraic model for $\Map(X,S^{2n};1)$ in Theorem \ref{component model} is not minimal, unlikely to the degree zero component in the previous section. Then we construct a minimal model for $\Map(X,S^{2n};1)_{(0)}$ from it. Let $W$ be the graded vector space as in \eqref{W}. We define a dga $(\Lambda W,d)$ by
		\begin{align}
			\label{differential 1}
			d(u\otimes x)&=0,\\
			\label{differential 1.2}
			d(v\otimes x)&=\sum_{x_1,x_2\in\B}\epsilon(x_1,x_2,x)(u\otimes x_1)(u\otimes x_2),\\
			\label{differential 1.3}d(v\otimes w)
			&=2(u\otimes 1)+\sum_{x_1,x_2\in\widehat{\B}}\epsilon(x_1,x_2,w)(u\otimes x_1)(u\otimes x_2)
		\end{align}
		for $x\in\B_0\sqcup\widehat{\B}$, which is different from the one in the previous section. Quite similarly to Corollary \ref{model k=0}, we can see that the dga $(\Lambda W,d)$ is isomorphic with the quotient dga  $(\Lambda(V\otimes H_*(X))/I_1,d)$ in Theorem \ref{component model}, and so we get:

		\begin{lemma}
			The dga $(\Lambda W,d)$ is an algebraic model for $\Map(X,S^{2n};1)_{(0)}$.
		\end{lemma}

		By definition, the dga $(\Lambda W,d)$ is not minimal, and so we construct a minimal model for $\Map(X,S^{2n};1)_{(0)}$ from it. Consider an element
		\[
		\eta=d(v\otimes w)-2(u\otimes 1)
		\]
		of $\Lambda W$. Then we have $d\eta=0$. We define
		\[
		v\odot x=(v\otimes w)(u\otimes x)-v\otimes x\text{\quad and\quad}v\odot1=v\otimes 1-\frac{1}{4}(v\otimes w)(2(u\otimes 1)-\eta)
		\]
		for $x\in\widehat{\B}$. Then by \eqref{differential 1}, \eqref{differential 1.2} and \eqref{differential 1.3}, we have
		\begin{align}
			\label{differential 2}
			d(v\odot x)&=\eta(v\otimes x)-\sum_{x_1,x_2\in\widehat{\B}}\epsilon(x_1,x_2,x)(u\otimes x_1)(u\otimes x_2),\\
			\label{differential 3}
			d(v\odot 1)&=\frac{1}{4}\eta^2
		\end{align}
		for $x\in\widehat{\B}$. Consider a vector subspace
		\[
		\overline{W}=\langle u\otimes x,v\odot x,v\odot 1\mid x\in\widehat{\B}\rangle
		\]
		of $\Lambda W$. Then since $\eta\in\Lambda\overline{W}$, it follows from \eqref{differential 1}, \eqref{differential 2} and \eqref{differential 3} that we get a subdga $(\Lambda\overline{W},d)$ of $(\Lambda W,d)$.

		\begin{proposition}
			\label{model 1}
			The dga $(\Lambda\overline{W},d)$ is a minimal model for $\Map(X,S^{2n};1)_{(0)}$.
		\end{proposition}

		\begin{proof}
			By definition, $(\Lambda\overline{W},d)$ is minimal, and so it remains to show that the inclusion $(\Lambda\overline{W},d)\to(\Lambda W,d)$ is a quasi-isomorphism. Consider a vector subspace
			\[
			\widehat{W}=\langle v\otimes w,2(u\otimes1)+\eta\rangle
			\]
			of $\Lambda W$. Then since $d(v\otimes w)=2(u\otimes1)+\eta$, we get a contractible subdga $(\Lambda\widehat{W},d)$ of $(\Lambda W,d)$, and so we get a dga map
			\[
			f\colon(\Lambda\overline{W},d)\otimes(\Lambda\widehat{W},d)\to(\Lambda W,d),\quad x\otimes y\mapsto xy.
			\]
			Clearly, $f_0$ is an isomorphism, hence so is $f$. Thus we may identify the inclusion $(\Lambda\overline{W},d)\to(\Lambda W,d)$ with the inclusion of $(\Lambda\overline{W},d)$ into the first factor of $(\Lambda\overline{W},d)\otimes(\Lambda\widehat{W},d)$, completing the proof.
		\end{proof}

		By Lemmas \ref{NE(minimal)} and \ref{1-conn}, we aim to compute the self-closeness number of $(\Lambda\overline{W},d)$ for determining that of $\Map(X,S^{2n};1)_{(0)}$. Let $U$ be the vector subspace of $\overline{W}$ spanned by elements of degree $\le 4n-2$. Then we have
		\[
		\overline{W}=U\oplus\langle v\odot 1\rangle.
		\]
		By \eqref{differential 1} and \eqref{differential 2}, we get a subdga $(\Lambda U,d)$ of $(\Lambda\overline{W},d)$.

		\begin{proposition}
			\label{split}
			There is an isomorphism
			\[
			(\Lambda U,d)\otimes(\Lambda(s),0)\xrightarrow{\cong}(\Lambda\overline{W},d)
			\]
			where $|s|=4n-1$.
		\end{proposition}

		Assuming Proposition \ref{split}, we determine $\NE(\Lambda\overline{W},d)$.

		\begin{proposition}
			\label{NE 3}
			$\NE(\Lambda\overline{W},d)=4n-1$.
		\end{proposition}
		
		\begin{proof}
			The largest degree of elements of $\overline{W}$ is $4n-1$, and so $\NE(\Lambda\overline{W},d)\le 4n-1$. By Proposition \ref{split}, it is easy to construct a dga map $f\colon(\Lambda\overline{W},d)\to(\Lambda\overline{W},d)$ such that $f_0\vert_U$ is an isomorphism but $f_0$ itself is not an isomorphism. Then $\NE(\Lambda\overline{W},d)\ge 4n-1$, completing the proof.
		\end{proof}

		Now we are ready to prove Theorem \ref{main}.
		
		\begin{proof}
			[Proof of Theorem \ref{main}]
			Combine Propositions \ref{NE 1}, \ref{NE 2} and \ref{NE 3}.
		\end{proof}

		It remains to prove Proposition \ref{split}. Suppose that there is a decomposable element $\zeta$ of $\Lambda U$ such that
		\begin{equation}
			\label{zeta}
			\frac{1}{4}\eta^2=d\zeta.
		\end{equation}
		Then we can define a dga map
		\[
		f\colon(\Lambda U,d)\otimes(\Lambda(s),0)\to(\Lambda\overline{W},d)
		\]
		by $f(x\otimes 1)=x$ for $x\in\Lambda U$ and $f(1\otimes s)=v\odot 1-\zeta$. Indeed, $df(x\otimes 1)=dx=f(dx\otimes 1)$ since $dx\in\Lambda U$, and
		\[
		ds=0=d(v\odot 1)-\frac{1}{4}\eta^2=d(v\odot 1-\zeta)
		\]
		by \eqref{differential 3}. Since $\zeta$ is decomposable, $f_0$ is an isomorphism, and so $f$ is an isomorphism too. Thus Proposition \ref{split} is proved by the following lemma.

		\begin{lemma}
			\label{zeta lemma}
			There is a decomposable element $\zeta$ of $\Lambda U$ satisfying \eqref{zeta}.
		\end{lemma}

		\subsection{The vector space $\mathcal{U}$}
		
		To prove Lemma \ref{zeta lemma}, we define a vector subspace
		\[
		\mathcal{U}=\langle(u\otimes x_1)(u\otimes x_2)(u\otimes x_3)\mid|x_1|+|x_2|+|x_3|=2n\rangle
		\]
		of $\Lambda U$. First, we will find an element $\xi$ of $\mathcal{U}$ and a decomposable element $\alpha$ of $\Lambda U$ such that
		\begin{equation}
			\label{alpha}
			\eta^2=\xi+d\alpha.
		\end{equation}
		For each $x\in\B_i\subset\widehat{\B}$, there is a unique element $\hat{x}\in\B$ such that $\epsilon(x,\hat{x},w)\ne 0$, where $|x|+|\hat{x}|=2n$. We abbreviate $\epsilon(x,\hat{x},w)$ by $\epsilon(x)$. Since $\hat{\hat{x}}=x$, we have
		\[
		\eta=\sum_{x\in\widehat{\B}}\epsilon(x)(u\otimes x)(u\otimes\hat{x}).
		\]
		We define
		\begin{equation}
			\label{xi}
			\xi=\sum_{x_1,x_2,x_3\in\widehat{\B}}\epsilon(x_3)\epsilon(x_1,x_2,x_3)(u\otimes x_1)(u\otimes x_2)(u\otimes\hat{x}_3).
		\end{equation}
		Then $\xi$ is a decomposable element of $\Lambda U$.

		\begin{lemma}
			\label{alpha xi}
			There is a decomposable element $\alpha\in\Lambda U$ satisfying \eqref{alpha}.
		\end{lemma}
		
		\begin{proof}
			Let
			\[
			\alpha=\sum_{x\in\widehat{\B}}\epsilon(x)(v\odot x)(u\otimes\hat{x})
			\]
			Then $\alpha$ is a decomposable element of $\Lambda U$, and by \eqref{differential 1} and \eqref{differential 2} we have
			\begin{align*}
				d\alpha&=d\left(\sum_{x\in\widehat{\B}}\epsilon(x)(v\odot x)(u\otimes\hat{x})\right)\\
				&=\sum_{x\in\widehat{\B}}\epsilon(x)(d(v\odot x))(u\otimes\hat{x})\\
				&=\sum_{x\in\widehat{\B}}\epsilon(x)\left(\eta(u\otimes x)(u\otimes\hat{x})-\sum_{x_1,x_2\in\widehat{\B}}\epsilon(x_1,x_2,x)(u\otimes x_1)(u\otimes x_2)(u\otimes\hat{x})\right)\\
				&=\eta^2-\xi.
			\end{align*}
			So the statement is proved.
		\end{proof}

		Next we define a decomposable element $\mu\in\Lambda U$ such that $d\mu\in\mathcal{U}$. Let
		\[
		\B_-=\{x\in\B\mid 0<|x|<n\}.
		\]
		For $x\in\B_-$, let
		
		\[
		\lambda(x)=\frac{3(n-|x|)}{n}\epsilon(\hat{x}),\quad\mu(x)=(-1)^{|x||\hat{x}|}(v\odot x)(u\otimes\hat{x})-(v\odot\hat{x})(u\otimes x)
		\]
		and let
		\[
		\mu=\sum_{x\in\B_-}\lambda(x)\mu(x).
		\]
		By \eqref{differential 1} and \eqref{differential 2} we have
		\begin{align}
			\label{dmu} d\mu(x)&=d\left((-1)^{|x||\hat{x}|}(v\odot x)(u\otimes\hat{x})-(v\odot\hat{x})(u\otimes x)\right)\\
			\notag&=(-1)^{|x||\hat{x}|}\left(d(v\odot x)\right)(u\otimes\hat{x})-\left(d(v\odot\hat{x})\right)(u\otimes x)\\
			\notag&=\sum_{x_1,x_2\in\widehat{\B}}\epsilon(x_1,x_2,\hat{x})(u\otimes x_1)(u\otimes x_2)(u\otimes x)\\
			\notag&\quad-(-1)^{|x||\hat{x}|}\sum_{x_1,x_2\in\widehat{\B}}\epsilon(x_1,x_2,x)(u\otimes x_1)(u\otimes x_2)(u\otimes\hat{x}).
		\end{align}
		Then it is easy to see that $d\mu(x)$ belongs to $\mathcal{U}$, hence so does $d\mu$ too.

		\begin{proposition}
			\label{mu}
			There is an equality
			\[
			\xi=d\mu.
			\]
		\end{proposition}

		We prove Lemma \ref{zeta lemma} by assuming Proposition \ref{mu}.

		\begin{proof}
			[Proof of Lemma \ref{zeta lemma}]
			Let $\zeta=\frac{1}{4}(\alpha+\mu)$. Then by Lemma \ref{alpha xi} and Proposition \ref{mu}, $\zeta$ is a decomposable element of $\Lambda U$ and satisfies \eqref{alpha}.
		\end{proof}

		We devote the rest of this paper to prove Proposition \ref{mu}. Since both $\xi$ and $d\mu$ belong to $\mathcal{U}$, we need to understand the structure of the vector space $\mathcal{U}$. We introduce a total order on $\widehat{\B}$. First, we equip $\B_i$ with any total order for each $0<i\le n$. Next, for each $n<i<2n$ we equip $\B_i$ with a total order such that $x<y$ if and only if $\hat{x}<\hat{y}$. Finally, we extend to a total order on $\widehat{\B}$ such that $|x|<|y|$ implies $x<y$. We write
		\[
		\B_-=\{\theta_1>\cdots>\theta_m\}.
		\]
		Then $\B_{n+1}\sqcup\cdots\sqcup\B_{2n-2}=\{\hat{\theta}_1,...,\hat{\theta}_m\}$. Let $\mathcal{F}_0=\B_n$, and for $1\le i\le m$, let $\mathcal{F}_i=\mathcal{F}_{i-1}\sqcup\{\theta_i,\hat{\theta}_i\}$. Then we get a filtration
		\[
		\B_n=\mathcal{F}_0\subset \mathcal{F}_1\subset\cdots\subset \mathcal{F}_m=\widehat{\B}.
		\]
		Since $\theta_m<\cdots<\theta_1<x$ for all $x\in\B_n\sqcup\cdots\sqcup\B_{2n-2}$, the least element of $\mathcal{F}_i$ is $\theta_i$ for $1\le i\le m$.

		We introduce a filtration of $\mathcal{U}$ by using the above filtration of $\widehat{\B}$. For $0\le i\le m$, we define a linear map $f_i\colon\mathcal{U}\to\mathcal{U}$ by
		\[
		f_i((u\otimes x_1)(u\otimes x_2)(u\otimes x_3))=
		\begin{cases}
			(u\otimes x_1)(u\otimes x_2)(u\otimes x_3)&x_1,x_2,x_3\in \mathcal{F}_i,\\
			0&\text{otherwise}.
		\end{cases}
		\]
		Let $\mathcal{U}_i=\mathrm{Im}\,f_i$ for $0\le i\le m$.
		\begin{lemma}
			The vector space $\mathcal{U}_0$ is trivial.
		\end{lemma}
		\begin{proof}
			Let $(u\otimes x_1)(u\otimes x_2)(u\otimes x_3)$ be an element of $\mathcal{U}$. By definition, $f_0((u\otimes x_1)(u\otimes x_2)(u\otimes x_3))\ne 0$ if and only if $x_1,x_2,x_3\in\mathcal{F}_0=\B_n$, implying $|x_1|+|x_2|+|x_3|=3n$. This is impossible because $|x_1|+|x_2|+|x_3|=2n$, and so the statement is proved.
		\end{proof}
		It is easy to see that $f_m$ is the identity map on $\mathcal{U}$ since $\mathcal{F}_m=\widehat{\B}$. Then we get a filtration
		\[
		0=\mathcal{U}_0\subset\mathcal{U}_1\subset\cdots\subset\mathcal{U}_m=\mathcal{U}.
		\]
		For $1\le i\le m$, let 
		\[
		\mathcal{V}_i=\langle(u\otimes x_1)(u\otimes x_2)(u\otimes\theta_i)\in\mathcal{U}\mid x_1,x_2\in\mathcal{F}_i\rangle
		\]
		be a vector subspace of $\mathcal{U}$.

		\begin{lemma}
			\label{U_i}
			There is an equality
			\[
			\mathcal{U}=\mathcal{V}_1\oplus\mathcal{V}_2\oplus\cdots\oplus\mathcal{V}_m.
			\]
		\end{lemma}
		
		\begin{proof}
			It is sufficient to show $\mathcal{U}_i=\mathcal{U}_{i-1}\oplus\mathcal{V}_i$ for $1\le i\le m$. By definition, the vector space $\mathcal{U}_i$ is the direct sum of $\mathcal{U}_{i-1}$ and a vector subspace of $\mathcal{U}$ spanned by $(u\otimes x_1)(u\otimes x_2)(u\otimes x_3)$ such that $x_1,x_2\in\mathcal{F}_i$ and $x_3=\theta_i$ or $\hat{\theta}_i$. If $x_3=\hat{\theta}_i$, then
			\[
			|x_1|+|x_2|=2n-|x_3|=|\theta_i|
			\]
			implying $x_1,x_2<\theta_i$. Then we get $x_1,x_2\not\in \mathcal{F}_i$, which is impossible. Clearly, the $x_3=\theta_i$ case is possible, completing the proof.
		\end{proof}

		

		\subsection{The derivation $\partial_i$}
		
		For $1\le i\le m$, let $p_i\colon\mathcal{U}\to\mathcal{V}_i$ denote the projection. Then by Lemma \ref{U_i}, in order to prove Proposition \ref{mu}, it is sufficient to show
		\begin{equation}
			\label{xi mu i}
			p_i(\xi-d\mu)=0
		\end{equation}
		for $1\le i\le m$. To this end, we introduce the derivation $\partial_i$. Let
		\[
		\widehat{U}=\langle u\otimes x\mid x\in\widehat{\B}\rangle.
		\]
		Then $\mathcal{U}$ is a vector subspace of $\Lambda\widehat{U}$. For $1\le i\le m$, we define a derivation $\partial_i\colon\Lambda\widehat{U}\to\Lambda\widehat{U}$ by
		\[
		\partial_i(u\otimes x)=
		\begin{cases}
			1&x=\theta_i,\\
			0&x\ne\theta_i
		\end{cases}
		\]
		for $x\in\widehat{\B}$ and the Leibuniz rule
		\begin{equation}
			\notag
			\partial_i(ab)=\partial_i(a)b+(-1)^{|a|}a\partial_i(b)
		\end{equation}
		for $a,b\in\Lambda\widehat{U}$. The following is immediate.
		\begin{lemma}
			\label{deri}
			For a non-trivial element $(u\otimes x_1)(u\otimes x_2)(u\otimes \theta_i)\in\mathcal{V}_i$, we have
			\begin{equation}
				\notag
				\partial_i((u\otimes x_1)(u\otimes x_2)(u\otimes\theta_i))=
				\begin{cases}
					(-1)^{|x_1|+|x_2|}(u\otimes x_1)(u\otimes x_2)&x_1,x_2\ne\theta_i,\\
					2(u\otimes x_1)(u\otimes\theta_i)&x_1\ne\theta_i,\,x_2=\theta_i,\\
					3(u\otimes\theta_i)^2&x_1,x_2=\theta_i.
				\end{cases}
			\end{equation}
		\end{lemma}
		\begin{proof}
			Since $(u\otimes x_1)(u\otimes x_2)(u\otimes\theta_i)\in\mathcal{V}_i$, we have $
			|x_1|+|x_2|+|\theta_i|=2n$.

			\noindent(1) For $x_1,x_2\ne\theta_i$,
			\begin{align*}
				\partial_i((u\otimes x_1)(u\otimes x_2)(u\otimes\theta_i))&=(-1)^{|u\otimes x_1|+|u\otimes x_2|}(u\otimes x_1)(u\otimes x_2)\partial_i(u\otimes\theta_i)\\
				&=(-1)^{|x_1|+|x_2|}(u\otimes x_1)(u\otimes x_2).
			\end{align*}

			\noindent(2) For $x\ne\theta_i,x_2=\theta_i$, $|u\otimes\theta_i|$ is even since $(u\otimes x_1)(u\otimes x_2)(u\otimes\theta_i)\ne0$. This also implies that $|x_1|$ is even since $|x_1|=2n-|x_2|-|\theta_i|=2n-2|\theta_i|$. Then we have
			\[
			\partial_i((u\otimes x_1)(u\otimes x_2)(u\otimes\theta_i))=(-1)^{|u\otimes x_1|}(u\otimes x_1)\partial_i\left((u\otimes\theta_i)^2\right)=2(u\otimes x_1)(u\otimes\theta_i).
			\]

			\noindent(3) For $x_1,x_2=\theta_i$, similarly to the above, we have $|\theta_i|$ even. Then we get
			\[
			\partial_i((u\otimes x_1)(u\otimes x_2)(u\otimes\theta_i))=3(u\otimes\theta_i)^2.
			\]
			Thus the proof is finished.
		\end{proof}
		The derivation $\partial_i$ has the following pleasant property.
		\begin{lemma}
			\label{derivation}
			For $1\le i\le m$, the derivation $\partial_i$ is injective on $\mathcal{V}_i\subset\Lambda\widehat{U}$.
		\end{lemma}
		
		\begin{proof}
			Note that $\mathfrak{B}_i=\{(u\otimes x_1)(u\otimes x_2)(u\otimes\theta_i)\in\mathcal{U}\mid x_1,x_2\in\mathcal{F}_i\}$ is a basis of $\mathcal{V}_i$. Then by Lemma \ref{deri}, $\partial_ia$ for $a\in\mathfrak{B}_i$ are linearly independent, proving the statement.
		\end{proof}

		In order to show \eqref{xi mu i}, it is sufficient to prove
		\begin{equation}
			\label{xi mu partial}
			\partial_i(p_i(\xi-d\mu))=0
		\end{equation}
		for $1\le i\le m$ by Lemma \ref{derivation}. So we describe $\partial_ip_i(\xi)$ and $\partial_ip_i(d\mu)$ respectively. To this end, we need the following properties of $\epsilon(x,y,z)$.

		\begin{lemma}
			\label{epsilon(x)}
			For $x\in\B_-$, we have
			\[
			\epsilon(x)=1\quad\text{and}\quad\epsilon(\hat{x})=(-1)^{|x||\hat{x}|}.
			\]
		\end{lemma}
		
		\begin{proof}
			The first equality follows immediately from the definition. Since $\hat{\hat{x}}=x$, we have
			\[
			\epsilon(\hat{x})=\epsilon(\hat{x},x,w)=(-1)^{|\hat{x}||x|}\epsilon(x,\hat{x},w)=(-1)^{|\hat{x}||x|}\epsilon(x)
			\]
			by Lemma \ref{comm}. Then we get the second equality.
		\end{proof}
		\begin{lemma}
			\label{degree}
			For $x_1,x_2,x_3\in\widehat{\B}$, we have the following:
			\begin{enumerate}
				\item $\epsilon(x_1,x_2,\hat{x}_3)(u\otimes x_1)(u\otimes x_2)(u\otimes x_3)=\epsilon(x_2,x_1,\hat{x}_3)(u\otimes x_2)(u\otimes x_1)(u\otimes x_3)$.
				\item For $1\le i\le m$, $p_i(\epsilon(x_1,x_2,\theta_i)(u\otimes x_1)(u\otimes x_2)(u\otimes\hat{\theta}_i))=0$.
				\item $\epsilon(\hat{x}_3)\epsilon(x_1,x_2,\hat{x}_3)=\epsilon(x_1)\epsilon(x_2,x_3,\hat{x}_1).$
				\item If at least two of $x_1,x_2,x_3$ are of the same degree and
				$$\epsilon(x_1,x_2,\hat{x}_3)(u\otimes x_1)(u\otimes x_2)(u\otimes x_3)\ne0,$$
				then $|x_1|,|x_2|,|x_3|$ are even.
			\end{enumerate}
		\end{lemma}
		\begin{proof}
			(1) The equality follows at once from Lemma \ref{comm}.

			\noindent(2) By \eqref{epsilon degree}, $|x_1|+|x_2|=|\theta_i|$ whenever $\epsilon(x_1,x_2,\theta_i)\ne0$, which implies that $x_1,x_2\notin\mathcal{F}_i$. Thus $p_i(\epsilon(x_1,x_2,\theta_i)(u\otimes x_1)(u\otimes x_2)(u\otimes\hat{\theta}_i))=0$ by definition.

			\noindent(3) By Lemma \ref{asso}, we have
			\begin{align*}
				\epsilon(\hat{x}_3)\epsilon(x_1,x_2,\hat{x}_3)&=\epsilon(\hat{x}_3,x_3,w)\epsilon(x_1,x_2,\hat{x}_3)\\
				&=\epsilon(x_1,\hat{x}_1,w)\epsilon(x_2,x_3,\hat{x}_1)\\
				&=\epsilon(x_1)\epsilon(x_2,x_3,\hat{x}_1)
			\end{align*}
			where $\hat{\hat{x}}_3=x_3$.

			\noindent(4) We only prove the $|x_1|=|x_2|$ case because other cases are proved quite similarly. Since $\epsilon(x_1,x_2,\hat{x}_3)\ne0$, we have $|x_1|+|x_2|=|\hat{x}_3|=2n-|x_3|$ by \eqref{epsilon degree}. Since $(u\otimes x_1)(u\otimes x_2)(u\otimes x_3)=(u\otimes x_1)^2(u\otimes x_3)\ne 0$, $|u\otimes x_1|=2n-|x_1|$ is even. Then we get that $|x_1|,|x_2|,|x_3|$ are even too.
		\end{proof}
		We prove a lemma which we are going to use.
		\begin{lemma}
			\label{x_1}
			For $x_2\in\mathcal{F}_{i-1}$, we have
			\begin{align*}
				&\partial_ip_i\left(\sum_{x_1,x_1'\in \widehat{\B}}\epsilon(x_1,x_1',\hat{x}_2)(u\otimes x_1)(u\otimes x_1')(u\otimes x_2)\right)\\
				&=2\sum_{x_1\in \mathcal{F}_i}\epsilon(\theta_i,x_1,\hat{x}_2)(u\otimes x_1)(u\otimes x_2).
			\end{align*}
		\end{lemma}
		\begin{proof}
			By definition, $p_i((u\otimes x_1)(u\otimes x_1')(u\otimes x_2))\ne 0$ if and only if $x_1,x_1'\in\mathcal{F}_i$ and at least one of $x_1,x_1'$ is $\theta_i$. Then we have
			\begin{align*}
				&p_i\left(\sum_{x_1,x_1'\in\widehat{\B}}\epsilon(x_1,x_1',\hat{x}_2)(u\otimes x_1)(u\otimes x_1')(u\otimes x_2)\right)\\
				&=\sum_{x_1,x_1'\in\mathcal{F}_i}\epsilon(x_1,x_1',\hat{x}_2)(u\otimes x_1)(u\otimes x_1')(u\otimes x_2)\\
				&=\epsilon(\theta_i,\theta_i,\hat{x}_2)(u\otimes\theta_i)^2(u\otimes x_2)\\
				&\quad+\sum_{x_1\in \mathcal{F}_{i-1}\sqcup\{\hat{\theta}_i\}}\epsilon(x_1,\theta_i,\hat{x}_2)(u\otimes x_1)(u\otimes\theta_i)(u\otimes x_2)\\
				&\quad+\sum_{x_1'\in \mathcal{F}_{i-1}\sqcup\{\hat{\theta}_i\}}\epsilon(\theta_i,x_1',\hat{x}_2)(u\otimes\theta_i)(u\otimes x_1')(u\otimes x_2).
			\end{align*}
			where the last equality holds because $\mathcal{F}_i=\mathcal{F}_{i-1}\sqcup\{\theta_i,\hat{\theta}_i\}$. By Lemma \ref{degree} (1),
			\begin{align*}
				&\sum_{x_1\in \mathcal{F}_{i-1}\sqcup\{\hat{\theta}_i\}}\epsilon(x_1,\theta_i,\hat{x}_2)(u\otimes x_1)(u\otimes\theta_i)(u\otimes x_2)\\
				&\quad+\sum_{x_1'\in \mathcal{F}_{i-1}\sqcup\{\hat{\theta}_i\}}\epsilon(\theta_i,x_1',\hat{x}_2)(u\otimes\theta_i)(u\otimes x_1')(u\otimes x_2)\\
				&=\sum_{x_1\in \mathcal{F}_{i-1}\sqcup\{\hat{\theta}_i\}}\epsilon(\theta_i,x_1,\hat{x}_2)(u\otimes\theta_i)(u\otimes x_1)(u\otimes x_2)\\
				&\quad+\sum_{x_1'\in \mathcal{F}_{i-1}\sqcup\{\hat{\theta}_i\}}\epsilon(\theta_i,x_1',\hat{x}_2)(u\otimes\theta_i)(u\otimes x_1')(u\otimes x_2)\\
				&=2\sum_{x_1\in \mathcal{F}_{i-1}\sqcup\{\hat{\theta}_i\}}\epsilon(\theta_i,x_1,\hat{x}_2)(u\otimes \theta_i)(u\otimes x_1)(u\otimes x_2).
			\end{align*}
			Then since $x_2\in\mathcal{F}_{i-1}$, we have
			\begin{align*}
				&\partial_ip_i\left(\sum_{x_1,x_1'\in\widehat{\B}}\epsilon(x_1,x_1',\hat{x}_2)(u\otimes x_1)(u\otimes x_1')(u\otimes x_2)\right)\\
				&=\partial_i\left(\epsilon(\theta_i,\theta_i,\hat{x}_2)(u\otimes\theta_i)^2(u\otimes x_2)\right)\\
				&\quad+2\partial_i\left(\sum_{x_1\in \mathcal{F}_{i-1}\sqcup\{\hat{\theta}_i\}}\epsilon(\theta_i,x_1,\hat{x}_2)(u\otimes \theta_i)(u\otimes x_1)(u\otimes x_2)\right)\\
				&=2\epsilon(\theta_i,\theta_i,\hat{x}_2)(u\otimes\theta_i)(u\otimes x_2)\\
				&\quad+2\sum_{x_1\in \mathcal{F}_{i-1}\sqcup\{\hat{\theta}_i\}}\epsilon(\theta_i,x_1,\hat{x}_2)(u\otimes \theta_i)(u\otimes x_1)(u\otimes x_2)\\
				&=2\sum_{x_1\in\mathcal{F}_i}\epsilon(\theta_i,x_1,\hat{x}_2)(u\otimes \theta_i)(u\otimes x_1)(u\otimes x_2)
			\end{align*}
			where the second equality holds by Lemma \ref{deri} and the last equality holds because $\mathcal{F}_i=\mathcal{F}_{i-1}\sqcup\{\theta_i,\hat{\theta}_i\}$. Thus the proof is finished.
		\end{proof}

		First, we describe $\partial_ip_i(\xi)$.

		\begin{lemma}
			\label{xi_i}
			For $1\le i\le m$, we have
			\[
			\partial_ip_i(\xi)=\epsilon(\hat{\theta}_i)\partial_ip_{i}(d\mu(\theta_i))+2\sum_{\substack{x_1\in \mathcal{F}_i\\x_2\in \mathcal{F}_{i-1}}}\epsilon(x_1,x_2,\hat{\theta}_i)(u\otimes x_1)(u\otimes x_2).
			\]
		\end{lemma}
		
		\begin{proof}
			By \eqref{xi} we have
			\begin{align*}
				\partial_ip_i(\xi)&=\partial_ip_i\left(\sum_{x_1,x_1',x_2\in\widehat{\B}}\epsilon(x_2)\epsilon(x_1,x_1',x_2)(u\otimes x_1)(u\otimes x_1')(u\otimes\hat{x}_2)\right)\\
				&=\partial_ip_i\left(\sum_{x_1,x_1',x_2\in\widehat{\B}}\epsilon(\hat{x}_2)\epsilon(x_1,x_1',\hat{x}_2)(u\otimes x_1)(u\otimes x_1')(u\otimes x_2)\right)\\
				&=\sum_{x_2\in\mathcal{F}_i}\left(\partial_ip_i\left(\sum_{x_1,x_1'\in\widehat{\B}}\epsilon(\hat{x}_2)\epsilon(x_1,x_1',\hat{x}_2)(u\otimes x_1)(u\otimes x_1')(u\otimes x_2)\right)\right)
			\end{align*}
			where the second equality holds because $\hat{\hat{x}}_2=x_2$ and the last equality holds by the definition of $p_i$. Since $\mathcal{F}_{i}=\mathcal{F}_{i-1}\sqcup\{\theta_i,\hat{\theta}_i\}$, we have
			\begin{align*}
				&\sum_{x_2\in\mathcal{F}_i}\left(\partial_ip_i\left(\sum_{x_1,x_1'\in\widehat{\B}}\epsilon(\hat{x}_2)\epsilon(x_1,x_1',\hat{x}_2)(u\otimes x_1)(u\otimes x_1')(u\otimes x_2)\right)\right)\\
				&=\sum_{x_2\in\mathcal{F}_{i-1}}\epsilon(\hat{x}_2)\left(\partial_ip_i\left(\sum_{x_1,x_1'\in\widehat{\B}}\epsilon(x_1,x_1',\hat{x}_2)(u\otimes x_1)(u\otimes x_1')(u\otimes x_2)\right)\right)\\
				&\quad+\epsilon(\hat{\theta}_i)\partial_ip_i\left(\sum_{x_1,x_1'\in \widehat{\B}}\epsilon(x_1,x_1',\hat{\theta}_i)(u\otimes x_1)(u\otimes x_1')(u\otimes\theta_i)\right)\\
				&\quad+\epsilon(\theta_i)\partial_ip_i\left(\sum_{x_1,x_1'\in \widehat{\B}}\epsilon(x_1,x_1',{\theta}_i)(u\otimes x_1)(u\otimes x_1')(u\otimes\hat{\theta}_i)\right)\\
				&=\sum_{x_2\in\mathcal{F}_{i-1}}\epsilon(\hat{x}_2)\left(\partial_ip_i\left(\sum_{x_1,x_1'\in\widehat{\B}}\epsilon(x_1,x_1',\hat{x}_2)(u\otimes x_1)(u\otimes x_1')(u\otimes x_2)\right)\right)\\
				&\quad+\epsilon(\hat{\theta}_i)\partial_ip_i\left(\sum_{x_1,x_1'\in \widehat{\B}}\epsilon(x_1,x_1',\hat{\theta}_i)(u\otimes x_1)(u\otimes x_1')(u\otimes\theta_i)\right)
			\end{align*}
			where the last equality holds by Lemma \ref{degree} (2). By Lemma \ref{x_1}, we have
			\begin{align*}
				&\sum_{x_2\in\mathcal{F}_{i-1}}\epsilon(\hat{x}_2)\left(\partial_ip_i\left(\sum_{x_1,x_1'\in\widehat{\B}}\epsilon(x_1,x_1',\hat{x}_2)(u\otimes x_1)(u\otimes x_1')(u\otimes x_2)\right)\right)\\
				&=2\sum_{\substack{x_1\in\mathcal{F}_i\\x_2\in\mathcal{F}_{i-1}}}\epsilon(\hat{x}_2)\epsilon(\theta_i,x_1,\hat{x}_2)(u\otimes x_1)(u\otimes x_2).
			\end{align*}
			By Lemmas \ref{epsilon(x)} and \ref{degree} (3), we have
			\[
			\epsilon(\hat{x}_2)\epsilon(\theta_i,x_1,\hat{x}_2)=\epsilon(\theta_i)\epsilon(x_1,x_2,\hat{\theta}_i)=\epsilon(x_1,x_2,\hat{\theta}_i)
			\]
			for $x_1,x_2\in\widehat{\B}$. This implies that
			\begin{align}
				\label{partial p}&\sum_{x_2\in\mathcal{F}_{i-1}}\epsilon(\hat{x}_2)\left(\partial_ip_i\left(\sum_{x_1,x_1'\in\widehat{\B}}\epsilon(x_1,x_1',\hat{x}_2)(u\otimes x_1)(u\otimes x_1')(u\otimes x_2)\right)\right)\\
				\notag&=2\sum_{\substack{x_1\in\mathcal{F}_i\\x_2\in\mathcal{F}_{i-1}}}\epsilon(x_1,x_2,\hat{\theta}_i)(u\otimes x_1)(u\otimes x_2).
			\end{align}
			On the other hand, by Lemma \ref{degree} (2), we get
			\begin{align}
				\label{p_idmu}&p_i\left(\sum_{x_1,x_1'\in \widehat{\B}}\epsilon(x_1,x_1',\hat{\theta}_i)(u\otimes x_1)(u\otimes x_1')(u\otimes\theta_i)\right)\\
				\notag&=p_i\left(\sum_{x_1,x_1'\in\widehat{\B}}\epsilon(x_1,x_1',\hat{\theta}_i)(u\otimes x_1)(u\otimes x_1')(u\otimes\theta_i)\right)\\
				\notag&\quad-(-1)^{|\theta_i||\hat{\theta}_i|}p_i\left(\sum_{x_1,x_1'\in\widehat{\B}}\epsilon(x_1,x_1',\theta_i)(u\otimes x_1)(u\otimes x_1')(u\otimes\hat{\theta}_i)\right)\\
				\notag&=p_i\left(d\mu(\theta_i)\right)
			\end{align}
			where the last equality holds by \eqref{dmu}. Thus the proof is finished by combining \eqref{partial p} and \eqref{p_idmu}.
		\end{proof}
		Hereafter, we write
		\begin{equation}
			\label{alp}
			\alpha(x_1,x_2)=\epsilon(x_1,x_2,\hat{\theta}_i)(u\otimes x_1)(u\otimes x_2)(u\otimes\theta_i).
		\end{equation}
		Let $\partial_ip_i(d\mu(\theta_i))=\sum_{x_1\le x_2\in \mathcal{F}_i}\hat{a}(x_1,x_2)(u\otimes x_1)(u\otimes x_2)$ for $\hat{a}(x_1,x_2)\in\Q$.
		\begin{lemma}
			\label{hat-a}
			There are equalities
			\[
			\hat{a}(x_1,x_2)=
			\begin{cases}
				\epsilon(x_1,x_2,\hat{\theta}_i)&\theta_i<x_1=x_2\in\mathcal{F}_i,\\
				4\epsilon(x_1,x_2,\hat{\theta}_i)&\theta_i=x_1<x_2\in\mathcal{F}_i,\\
				(-1)^{|x_1|+|x_2|}2\epsilon(x_1,x_2,\hat{\theta}_i)&\theta_i<x_1<x_2\in \mathcal{F}_i,\\
				3\epsilon(x_1,x_2,\hat{\theta}_i)&\theta_i=x_1=x_2\in\mathcal{F}_i.
			\end{cases}
			\]
		\end{lemma}
		
		\begin{proof}
			We have
			\begin{align*}
				p_i(d\mu(\theta_i))&=\sum_{x_1,x_2\in \mathcal{F}_i}\alpha(x_1,x_2)\\
				&=\sum_{\theta_i<x_1=x_2\in \mathcal{F}_i}\alpha(x_1,x_2)+\sum_{\theta_i<x_1<x_2\in \mathcal{F}_i}(\alpha(x_1,x_2)+\alpha(x_2,x_1))\\
				&\quad+\sum_{\theta_i=x_1<x_2\in \mathcal{F}_i}(\alpha(x_1,x_2)+\alpha(x_2,x_1))+\sum_{\theta_i=x_1=x_2\in \mathcal{F}_i}\alpha(x_1,x_2).
			\end{align*}
			By Lemma \ref{degree} (1), we have $\alpha(x_1,x_2)=\alpha(x_2,x_1)$. By Lemma \ref{deri}, we can make the following calculations, in which it suffices to assume $\alpha(x_1,x_2)\ne0$.

			\noindent(1) For $\theta_i<x_1=x_2$, $|x_1|,|x_2|,|\theta_i|$ are even by Lemma \ref{degree} (4). Then we have
			\begin{align*}
				\partial_i\alpha(x_1,x_2)&=(-1)^{|x_1|+|x_2|}\epsilon(x_1,x_2,\hat{\theta}_i)(u\otimes x_1)(u\otimes x_2)\\
				&=\epsilon(x_1,x_2,\hat{\theta}_i)(u\otimes x_1)(u\otimes x_2).
			\end{align*}

			\noindent(2) For $\theta_i=x_1<x_2$, $|x_1|,|x_2|,|\theta_i|$ are even by Lemma \ref{degree} (4). Then we have
			\begin{align*}
				\partial_i(\alpha(x_1,x_2)+\alpha(x_2,x_1))&=2\partial_i(\alpha(x_1,x_2))\\
				&=2\epsilon(x_1,x_2,\hat{\theta}_i)\partial_i\left((u\otimes\theta_i)^2(u\otimes x_2)\right)\\
				&=4\epsilon(x_1,x_2,\hat{\theta}_i)(u\otimes\theta_i)(u\otimes x_2)\\
				&=4\epsilon(x_1,x_2,\hat{\theta}_i)(u\otimes x_1)(u\otimes x_2).
			\end{align*}

			\noindent(3) For $\theta_i<x_1<x_2$,
			\begin{align*}
				\partial_i(\alpha(x_1,x_2)+\alpha(x_2,x_1))&=2\partial_i(\alpha(x_1,x_2))\\
				&=(-1)^{|x_1|+|x_2|}2\epsilon(x_1,x_2,\hat{\theta}_i)(u\otimes x_1)(u\otimes x_2).
			\end{align*}

			\noindent(4) For $\theta_i=x_1=x_2$,
			\[
			\partial_i\alpha(x_1,x_2)=3\epsilon(x_1,x_2,\hat{\theta}_i)(u\otimes x_1)(u\otimes x_2).
			\]
			Thus the statement is proved.
		\end{proof}
		Note that we can write
		\[
		\partial_ip_i(\xi)=\sum_{x_1\le x_2\in \mathcal{F}_i}a(x_1,x_2)(u\otimes x_1)(u\otimes x_2).
		\]
		for $a(x_1,x_2)\in\Q$. We compute $a(x_1,x_2)$ explicitly in order to compare them with the coefficients of $(u\otimes x_1)(u\otimes x_2)$ in $\partial_ip_i(d\mu)$ later, for all $x_1\le x_2\in\mathcal{F}_i$.

		\begin{proposition}
			\label{a}
			For $x_1,x_2\in \mathcal{F}_i$, we have
			\[
			a(x_1,x_2)=
			\begin{cases}
				3\epsilon(x_1,x_2,\hat{\theta}_i)&\theta_i<x_1=x_2,\\
				6\epsilon(x_1,x_2,\hat{\theta}_i)&\theta_i=x_1<x_2,\\
				6\epsilon(x_1,x_2,\hat{\theta}_i)&\theta_i<x_1<x_2,\\
				3\epsilon(x_1,x_2,\hat{\theta}_i)&\theta_i=x_1=x_2.
			\end{cases}
			\]
		\end{proposition}
		
		\begin{proof}
			By Lemmas \ref{derivation}, \ref{xi_i} and \ref{hat-a}, $p_i(\xi)$ is a linear combination of $\alpha(x_1,x_2)$, where $\alpha(x_1,x_2)$ is as in \eqref{alp}. Thus it suffices to assume $\alpha(x_1,x_2)\ne0$. By Lemma \ref{deri}, we can make the following calculations.

			\noindent(1) For $\theta_i<x_1=x_2$, $|\theta_i|$ is even by Lemma \ref{degree} (4). Then since $\epsilon(\hat{\theta}_i)=1$ by Lemma \ref{epsilon(x)}, we have
			\[
			a(x_1,x_2)=\epsilon(\hat{\theta}_i)\hat{a}(x_1,x_2)+2\epsilon(x_1,x_2,\hat{\theta}_i)=3\epsilon(x_1,x_2,\hat{\theta}_i).
			\]

			\noindent(2) For $\theta_i=x_1<x_2$, similarly to the above, we get $\epsilon(\hat{\theta}_i)=1$. Then we have
			\[
			a(x_1,x_2)=\epsilon(\hat{\theta}_i)\hat{a}(x_1,x_2)+2\epsilon(x_1,x_2,\hat{\theta}_i)=6\epsilon(x_1,x_2,\hat{\theta}_i).
			\]

			\noindent(3) For $\theta_i<x_1<x_2$, we get $\epsilon(\hat{\theta}_i)=(-1)^{|\theta_i||\hat{\theta}_i|}=(-1)^{|\hat{\theta}_i|}=(-1)^{|x_1|+|x_2|}$ by Lemma \ref{epsilon(x)} and \eqref{epsilon degree}. Then by Lemma \ref{comm}, we have
			\begin{align*}
				a(x_1,x_2)&=\epsilon(\hat{\theta}_i)\hat{a}(x_1,x_2)+2\epsilon(x_1,x_2,\hat{\theta}_i)+(-1)^{|u\otimes x_1||u\otimes x_2|}2\epsilon(x_2,x_1,\hat{\theta}_i)\\
				&=(2((-1)^{|x_1|+|x_2|})^2+2+2)\epsilon(x_1,x_2,\hat{\theta}_i)\\
				&=6\epsilon(x_1,x_2,\hat{\theta}_i).
			\end{align*}

			\noindent(4) For $\theta_i=x_1=x_2$, similarly to the $\theta_i<x_1=x_2$ case, we get $\epsilon(\hat{\theta}_i)=1$. Then we have
			\[
			a(x_1,x_2)=\epsilon(\hat{\theta}_i)\hat{a}(x_1,x_2)=3\epsilon(x_1,x_2,\hat{\theta}_i).
			\]
			Thus the proof is finished.
		\end{proof}

		Next, we describe $\partial_ip_i(d\mu)$. We extend the definition of $\lambda(x)$ to $\widehat{\B}-\B_-$ as follows. Let $\lambda(x)=-\lambda(\hat{x})$ for $n<|x|<2n$ and $\lambda(x)=0$ for $|x|=n$.

		\begin{lemma}
			\label{partial-mu}
			For $1\le i\le m$, we have
			\begin{align*}
				\partial_ip_i(d\mu)&=\lambda(\theta_i)\partial_ip_i(d\mu(\theta_i))\\
				&\quad+2\sum_{\substack{x_1\in \mathcal{F}_i\\x_2\in \mathcal{F}_{i-1}}}(-1)^{|x_2||\hat{x}_2|}\lambda(x_2)\epsilon(x_1,x_2,\hat{\theta}_i)(u\otimes x_1)(u\otimes x_2).
			\end{align*}
		\end{lemma}
		
		\begin{proof}
			By definition, we have
			\begin{align*}
				\partial_ip_i(d\mu)&=\partial_ip_i\left(\sum_{x_2\in\B_-}\lambda(x_2)d\mu(x_2)\right)\\
				&=\sum_{x_2\in\mathcal{F}_i\cap\B_-}\lambda(x_2)\partial_ip_i\left(d\mu(x_2)\right)\\
				&=\lambda(\theta_i)\partial_ip_i(d\mu(\theta_i))+\sum_{x_2\in \mathcal{F}_{i-1}\cap\B_-}\lambda(x_2)\partial_ip_i(d\mu(x_2)).
			\end{align*}
			where the second equality holds by the definition of $p_i$ and the last equality holds because $\mathcal{F}_i\cap\B_-=(\mathcal{F}_{i-1}\cap\B_-)\sqcup\{\theta_i\}$. Let $x_2\in \mathcal{F}_{i-1}\cap\B_-$, and let
			\begin{align*}
				A(x_2)&=\sum_{x_1,x_1'\in\widehat{\B}}\epsilon(x_1,x_1',\hat{x}_2)(u\otimes x_1)(u\otimes x_1')(u\otimes x_2),\\
				B(x_2)&=\sum_{x_1,x_1'\in\widehat{\B}}\epsilon(x_1,x_1',x_2)(u\otimes x_1)(u\otimes x_1')(u\otimes\hat{x}_2).
			\end{align*}
			Then $d\mu(x_2)=A(x_2)-(-1)^{|x_2||\hat{x}_2|}B(x_2)$ by \eqref{dmu}. Since $x_2,\hat{x}_2\in\mathcal{F}_{i-1}$, by Lemma \ref{x_1}, we have
			\begin{align*}
				\partial_ip_i(A(x_2))&=2\sum_{x_1\in\mathcal{F}_i}\epsilon(\theta_i,x_1,\hat{x}_2)(u\otimes x_1)(u\otimes x_2),\\
				\partial_ip_i(B(x_2))&=2\sum_{x_1\in\mathcal{F}_i}\epsilon(\theta_i,x_1,x_2)(u\otimes x_1)(u\otimes \hat{x}_2).
			\end{align*}
			Then we have
			\begin{align*}
				\partial_ip_i(d\mu(x_2))&=\partial_ip_i\left(A(x_2)-(-1)^{|x_2||\hat{x}_2|}B(x_2)\right)\\
				&=2\sum_{x_1\in \mathcal{F}_i}\left(\epsilon(\theta_i,x_1,\hat{x}_2)(u\otimes x_1)(u\otimes x_2)\right.\\
				&\quad-\left.(-1)^{|x_2||\hat{x}_2|}\epsilon(\theta_i,x_1,x_2)(u\otimes x_1)(u\otimes \hat{x}_2)\right).
			\end{align*}
			By Lemmas \ref{epsilon(x)} and \ref{degree} (3), we have
			\begin{align*}
				\epsilon(\theta_i,x_1,\hat{x}_2)&=\frac{\epsilon(\theta_i)}{\epsilon(\hat{x}_2)}\epsilon(x_1,x_2,\hat{\theta}_i)=(-1)^{|x_2||\hat{x}_2|}\epsilon(x_1,x_2,\hat{\theta}_i),\\
				\epsilon(\theta_i,x_1,x_2)&=\frac{\epsilon(\theta_i)}{\epsilon(x_2)}\epsilon(x_1,\hat{x}_2,\hat{\theta}_i)=\epsilon(x_1,\hat{x}_2,\hat{\theta}_i).
			\end{align*}
			Then we have
			\begin{align*}
				\partial_ip_i(d\mu(x_2))&=2\sum_{x_1\in \mathcal{F}_i}(-1)^{|x_2||\hat{x}_2|}\left(\epsilon(x_1,x_2,\hat{\theta}_i)(u\otimes x_1)(u\otimes x_2)\right.\\
				&\quad-\left.\epsilon(x_1,\hat{x}_2,\hat{\theta}_i)(u\otimes x_1)(u\otimes \hat{x}_2)\right).
			\end{align*}
			Write $\beta(x_1,x_2)=(-1)^{|x_2||\hat{x}_2|}\epsilon(x_1,x_2,\hat{\theta}_i)(u\otimes x_1)(u\otimes x_2)$. Then we get
			\begin{align}
				\label{p(dmu)}&\sum_{x_2\in\mathcal{F}_{i-1}\cap\B_-}\lambda(x_2)\partial_ip_i(d\mu(x_2))\\
				\notag&=2\sum_{x_2\in\mathcal{F}_{i-1}\cap\B_-}\left(\lambda(x_2)\sum_{x_1\in\mathcal{F}_i}(-1)^{|x_2||\hat{x}_2|}\epsilon(x_1,x_2,\hat{\theta}_i)(u\otimes x_1)(u\otimes x_2)\right.\\
				\notag&\quad-\left.\lambda(x_2)\sum_{x_1\in\mathcal{F}_i}(-1)^{|\hat{x}_2||x_2|}\epsilon(x_1,\hat{x}_2,\hat{\theta}_i)(u\otimes x_1)(u\otimes\hat{x}_2)\right)\\
				\notag&=2\sum_{\substack{x_1\in\mathcal{F}_i\\x_2\in\mathcal{F}_{i-1}\cap\B_-}}\left(\lambda(x_2)\beta(x_1,x_2)-\lambda(x_2)\beta(x_1,\hat{x}_2)\right).
			\end{align}
			Let 
			\begin{align*}
				C&=\sum_{\substack{x_1\in\mathcal{F}_i\\x_2\in\mathcal{F}_{i-1}\cap\B_-}}\lambda(x_2)\beta(x_1,x_2),\\ D&=\sum_{\substack{x_1\in\mathcal{F}_i\\\hat{x}_2\in\mathcal{F}_{i-1}\cap\B_-}}\lambda(x_2)\beta(x_1,x_2),\\ E&=\sum_{\substack{x_1\in\mathcal{F}_i\\x_2\in\B_n}}\lambda(x_2)\beta(x_1,x_2).
			\end{align*}
			Since $\hat{\hat{x}}_2=x_2$ and $\lambda(\hat{x}_2)=-\lambda(x_2)$,
			\begin{align*}
				-\sum_{\substack{x_1\in\mathcal{F}_i\\x_2\in\mathcal{F}_{i-1}\cap\B_-}}\lambda(x_2)\beta(x_1,\hat{x}_2)&=\sum_{\substack{x_1\in\mathcal{F}_i\\x_2\in\mathcal{F}_{i-1}\cap\B_-}}\lambda(\hat{x}_2)\beta(x_1,\hat{x}_2)\\
				&=\sum_{\substack{x_1\in\mathcal{F}_i\\\hat{x}_2\in\mathcal{F}_{i-1}\cap\B_-}}\lambda(x_2)\beta(x_1,x_2)\\
				&=D.
			\end{align*}
			Since $\lambda(x_2)=0$ for $x_2\in\B_n$, we have $E=0$. it follows from \eqref{p(dmu)} that
			\begin{align*}
				&=\sum_{x_2\in\mathcal{F}_{i-1}\cap\B_-}\lambda(x_2)\partial_ip_i(d\mu(x_2))\\
				&=2\sum_{\substack{x_1\in\mathcal{F}_i\\x_2\in\mathcal{F}_{i-1}\cap\B_-}}\lambda(x_2)\beta(x_1,x_2)+2\sum_{\substack{x_1\in\mathcal{F}_i\\\hat{x}_2\in\mathcal{F}_{i-1}\cap\B_-}}\lambda(x_2)\beta(x_1,x_2)\\
				&\quad+2\sum_{\substack{x_1\in\mathcal{F}_i\\x_2\in\B_n}}\lambda(x_2)\beta(x_1,x_2)\\
				&=2C+2D+2E.
			\end{align*}
			On the other hand, since
			\[
			\mathcal{F}_{i-1}=(\mathcal{F}_{i-1}\cap\B_-)\sqcup\B_n\sqcup\{x_2\in\widehat{\B}\mid\hat{x}_2\in\mathcal{F}_{i-1}\cap\B_-\},
			\]
			we have
			\begin{align*}
				C+D+E&=\sum_{\substack{x_1\in \mathcal{F}_i\\x_2\in \mathcal{F}_{i-1}}}\lambda(x_2)\beta(x_1,x_2)\\
				&=\sum_{\substack{x_1\in \mathcal{F}_i\\x_2\in \mathcal{F}_{i-1}}}(-1)^{|x_2||\hat{x}_2|}\lambda(x_2)\epsilon(x_1,x_2,\hat{\theta}_i)(u\otimes x_1)(u\otimes x_2).
			\end{align*}
			Thus the proof is finished.
		\end{proof}

		We show a property of $\lambda(x)$ that we are going to use.

		\begin{lemma}
			\label{lambda}
			If $x_1,x_2,x_3\in\widehat{\B}$ satisfy $|x_1|+|x_2|+|x_3|=2n$, then
			\[
			\sum_{i=1}^3\epsilon(\hat{x}_i)\epsilon(x_i)\lambda(x_i)=3.
			\]
		\end{lemma}
		
		\begin{proof}
			We may assume $|x_1|\le|x_2|\le|x_3|$. Let $x\in\widehat{\B}$. By Lemma \ref{epsilon(x)}, we have that for $0<|x|<n$,
			\[
			\epsilon(\hat{x})\epsilon(x)\lambda(x)=\epsilon(\hat{x})^2\epsilon(x)\frac{3(n-|x|)}{n}=\frac{3(n-|x|)}{n}
			\]
			and for $n<|x|<2n$, we have
			\[
			\epsilon(\hat{x})\epsilon(x)\lambda(x)=-\epsilon(\hat{x})\epsilon(x)^2\frac{3(n-|\hat{x}|)}{n}=-\frac{3(n-2n+|x|)}{n}=\frac{3(n-|x|)}{n}.
			\]
			Then for $|x_1|\le|x_2|\le|x_3|<n$, we have
			\[
			\sum_{i=1}^3\epsilon(\hat{x}_i)\epsilon(x_i)\lambda(x_i)=\frac{3}{n}(3n-|x_1|-|x_2|-|x_3|)=3.
			\]
			For $|x_1|\le|x_2|\le|x_3|=n$, we have $|x_1|+|x_2|=n$ and $\lambda(x_3)=0$. Then
			\[
			\sum_{i=1}^3\epsilon(\hat{x}_i)\epsilon(x_i)\lambda(x_i)=\frac{3}{n}(2n-|x_1|-|x_2|)=3.
			\]
			For $|x_1|\le|x_2|<n<|x_3|$, we have
			\[
			\sum_{i=1}^3\epsilon(\hat{x}_i)\epsilon(x_i)\lambda(x_i)=\frac{3}{n}(3n-|x_1|-|x_2|-|x_3|)=3.
			\]
			Thus the proof is finished.
		\end{proof}

		Now we are ready to prove Proposition \ref{mu}.

		\begin{proof}
			[Proof of Proposition \ref{mu}]
			As mentioned above, it is sufficient to prove \eqref{xi mu partial} for $1\le i\le m$. We can write
			\[
			\partial_ip_i(d\mu)=\sum_{x_1\le x_2\in \mathcal{F}_i}b(x_1,x_2)(u\otimes x_1)(u\otimes x_2)
			\]
			for $b(x_1,x_2)\in\Q$. Then we aim to show $a(x_1,x_2)=b(x_1,x_2)$ for all $x_1\le x_2\in \mathcal{F}_i$, where $a(x_1,x_2)$ is as in Proposition \ref{a}. By Lemmas \ref{derivation}, \ref{hat-a} and \ref{partial-mu}, $p_i(d\mu)$ is a linear combination of $\alpha(x_1,x_2)$ where $\alpha(x_1,x_2)$ is as in \eqref{alp}. Thus it suffices to assume $\alpha(x_1,x_2)\ne0$. Since $|x_1|+|x_2|+|\theta_i|=2n$ by \eqref{epsilon degree}, we have
			\[
			\epsilon(x_1)\epsilon(\hat{x}_1)\lambda(x_1)+\epsilon(x_2)\epsilon(\hat{x}_2)\lambda(x_2)+\epsilon(\theta_i)\epsilon(\hat{\theta}_i)\lambda({\theta}_i)=3
			\]
			by Lemma \ref{lambda}. Then by Lemma \ref{deri}, we can make the following calculations.

			\noindent(1) For $\theta_i<x_1=x_2$, $x_1,x_2,\theta_i\notin\B_n$ since $|x_1|+|x_2|+|\theta_i|=2n$. Since $|x_1|,|x_2|,|\theta_i|$ are even by Lemma \ref{degree} (4), $\epsilon(x)\epsilon(\hat{x})\lambda(x)=\lambda(x)$ for $x=x_1,x_2,\theta_i$ by Lemma \ref{epsilon(x)}. Then we get
			\begin{align*}
				b(x_1,x_2)&=\lambda(\theta_i)\hat{a}(x_1,x_2)+2\lambda(x_2)\epsilon(x_1,x_2,\hat{\theta}_i)\\
				&=(\epsilon(\hat{x}_1)\epsilon(x_1)\lambda(x_1)+\epsilon(\hat{x}_2)\epsilon(x_2)\lambda(x_2)+\epsilon(\hat{\theta}_i)\epsilon(\theta_i)\lambda(\theta_i))\epsilon(x_1,x_2,\hat{\theta}_i)\\
				&=3\epsilon(x_1,x_2,\hat{\theta}_i).
			\end{align*}

			\noindent(2) For $\theta_i=x_1<x_2$, $\theta_i=x_1\notin\B_n$ since $|x_1|+|x_2|+|\theta_i|=2n$. If $x_2\notin\B_n$, similarly to the above, we have $\epsilon(x)\epsilon(\hat{x})\lambda(x)=\lambda(x)$ for $x=x_1,x_2,\theta_i$. If $x_2\in\B_n$, we also have $\epsilon(x)\epsilon(\hat{x})\lambda(x)=\lambda(x)$ for $x=x_1,x_2,\theta_i$ since $\lambda(x_2)=0$. Then we get
			\begin{align*}
				b(x_1,x_2)&=\lambda(\theta_i)\hat{a}(x_1,x_2)+2\lambda(x_2)\epsilon(x_1,x_2,\hat{\theta}_i)\\
				&=2(\epsilon(\hat{x}_1)\epsilon(x_1)\lambda(x_1)+\epsilon(\hat{x}_2)\epsilon(x_2)\lambda(x_2)+\epsilon(\hat{\theta}_i)\epsilon(\theta_i)\lambda(\theta_i))\epsilon(x_1,x_2,\hat{\theta}_i)\\
				&=6\epsilon(x_1,x_2,\hat{\theta}_i).
			\end{align*}

			\noindent(3) For $\theta_i<x_1<x_2$, we have
			\begin{align*}
				&(-1)^{|x_2||\hat{x}_2|}2\lambda(x_2)\epsilon(x_1,x_2,\hat{\theta}_i)+(-1)^{|u\otimes x_1||u\otimes x_2|+|x_1||\hat{x}_1|}2\lambda(x_1)\epsilon(x_2,x_1,\hat{\theta}_i)\\
				&=(-1)^{|x_2||\hat{x}_2|}2\lambda(x_2)\epsilon(x_1,x_2,\hat{\theta}_i)+(-1)^{|x_1||\hat{x}_1|}2\lambda(x_1)\epsilon(x_1,x_2,\hat{\theta}_i)\\
				&=2(\epsilon(\hat{x}_1)\epsilon(x_1)\lambda(x_1)+\epsilon(\hat{x}_2)\epsilon(x_2)\lambda(x_2))\epsilon(x_1,x_2,\hat{\theta}_i)
			\end{align*}
			by Lemmas \ref{comm} and \ref{epsilon(x)}. Then we get
			\begin{align*}
				b(x_1,x_2)&=2(\epsilon(\hat{x}_1)\epsilon(x_1)\lambda(x_1)+\epsilon(\hat{x}_2)\epsilon(x_2)\lambda(x_2))\epsilon(x_1,x_2,\hat{\theta}_i)+\lambda(\theta_i)\hat{a}(x_1,x_2)\\
				&=2(\epsilon(\hat{x}_1)\epsilon(x_1)\lambda(x_1)+\epsilon(\hat{x}_2)\epsilon(x_2)\lambda(x_2)+\epsilon(\hat{\theta}_i)\epsilon(\theta_i)\lambda(\theta_i))\epsilon(x_1,x_2,\hat{\theta}_i)\\
				&=6\epsilon(x_1,x_2,\hat{\theta}_i).
			\end{align*}

			\noindent(4) For $\theta_i=x_1=x_2$, $|\theta_i|$ is even and $3|\theta_i|=2n$ by \eqref{epsilon degree} and Lemma \ref{degree} (4) respectively. Then we get
			\[
			\lambda(\theta_i)=\frac{3(n-\frac{2n}{3})}{n}=1\text{\quad and\quad}b(x_1,x_2)=\lambda(\theta_i)\hat{a}(x_1,x_2)=3\epsilon(x_1,x_2,\hat{\theta}_i).
			\]
			Thus by Proposition \ref{a}, we obtain $a(x_1,x_2)=b(x_1,x_2)$ for all $x_1\le x_2\in\mathcal{F}_i$, completing the proof.
		\end{proof}

	\end{document}